\documentclass[reqno]{amsart} 
\usepackage{amsfonts, amsmath, amsthm, amssymb, latexsym, graphicx, geometry, xcolor, colortbl,  mathtools}

\DeclareMathOperator{\inte}{int}

\def\cth{{\rm coth}\>}

\theoremstyle{plain}
\newtheorem{theorem}{Theorem}[section]

\newtheorem{lemma}[theorem]{Lemma}

\theoremstyle{definition}

\newtheorem{remark}[theorem]{Remark}  

\newtheorem{example}[theorem]{Example}

\numberwithin{equation}{section}

\title[Wigner- and Marchenko-Pastur-type limits for Jacobi processes]{Wigner- and Marchenko-Pastur-type limit
  theorems for Jacobi processes}

\author{Martin Auer, Michael Voit, Jeannette H.C. Woerner} 
\address{Fakult\"at Mathematik, Technische Universit\"at Dortmund,
          Vogelpothsweg 87,
          D-44221 Dortmund, Germany}
\email{martin.auer@math.tu-dortmund.de, michael.voit@math.tu-dortmund.de,\newline
 jeannette.woerner@math.tu-dortmund.de}


\begin{document}
\date{\today}

\begin{abstract}
  We study  Jacobi processes  $(X_{t})_{t\ge0}$  on the compact spaces  $[-1,1]^N$  and on the noncompact spaces $[1,\infty[^N$
      which are motivated by the Heckman-Opdam theory for the root systems of type BC and associated integrable particle systems.
      These processes
  depend on  three  positive parameters and degenerate in the freezing limit to solutions of
   deterministic dynamical systems. In the compact case, these models tend for $t\to\infty$
  to the distributions of the $\beta$-Jacobi ensembles and,
  in the freezing case, to  vectors consisting of  ordered zeros of 
  one-dimensional Jacobi polynomials.

  Representing these  processes by stochastic differential equations, we  derive almost sure analogues of Wigner's semicircle  and  Marchenko-Pastur
  limit laws for $N\to\infty$ for the empirical distributions of the $N$ particles on some local scale.
  We there allow for arbitrary initial conditions, which enter the limiting distributions via free convolutions.
  These results generalize corresponding stationary limit results in the compact case for
 $\beta$-Jacobi ensembles  and, in the deterministic case,  for 
the empirical distributions of the ordered zeros of Jacobi polynomials by Dette and Studden.
The results are also related to free limit theorems for multivariate Bessel processes, $\beta$-Hermite and  $\beta$-Laguerre
ensembles, and the asymptotic empirical distributions of the zeros of  Hermite and Laguerre polynomials for  $N\to\infty$.
\end{abstract}

\maketitle

\section{Introduction}

By classical results, the empirical distributions of $\beta$-Hermite, $\beta$-Laguerre, and $\beta$-Jacobi ensembles of
dimension $N$
tend for $N\to\infty$ to semicircle, Marchenko-Pastur  as well as  Kesten-McKay
and Wachter distributions respectively after suitable normalizations; see e.g.~\cite{CC, DN, J, RS, W} and references therein. 
Moreover, in the Hermite and Laguerre cases, there are dynamical versions of these results in terms of Bessel processes $(X_t^N)_{t\ge0}$
of dimension $N$ for the root systems of types A and B; see \cite{CGY, RV1} for the background on these processes. Namely,
let $\mu$ be some starting distribution on $\mathbb R$ or $[0,\infty[$, and let for $N\in\mathbb N$, $x_N$ be starting vectors in
    $\mathbb R^N$ such that the  empirical distributions of the components of the $x_N$ tend to $\mu$. If we consider the 
 Bessel processes $(X_t^N)_{t\ge0}$ with start in these points $x_N$, then under mild additional conditions and with an appropriate scaling, the 
 empirical distributions of the components of the $X_t^N$ tend for $N\to\infty$ almost surely weakly to measures $\mu_t$ for all $t\ge0$.
 In the  $\beta$-Hermite case, i.e., for  Bessel processes of type A, one has $\mu_t=\mu\boxplus\mu_{sc, 2\sqrt{t}}$, where
 $\mu_{sc, 2\sqrt{t}}$ denotes the semicircle distribution with radius $2\sqrt{t}$ and $\boxplus$ the usual additive free convolution;
 see Section 4.3 of \cite{AGZ} and \cite{VW1} for different approaches.
 Moreover, for  the  $\beta$-Laguerre case, i.e. for  Bessel processes of type B,
 there are corresponding results for $\mu_t$ in terms of Marchenko-Pastur distributions and a more complicated construction
 involving the usual additive free convolution in \cite{VW1}. This construction may also be described in terms of 
 the rectangular free convolutions of Benaych-Georges \cite{B1, B2}.
 Furthermore, these results for Bessel processes of types A and B can be transferred
 to stationary Ornstein-Uhlenbeck-type versions of these processes as indicated in the end of Section 2 of \cite{VW1}. For the background on stochastic analysis we recommend the monographs \cite{P, RW}, and\cite {AGZ, NS}
for free probability in our context.

 In this paper we  show that Ornstein-Uhlenbeck-type limit results also appear
 for certain $N$-dimensional Jacobi processes on $[-1,1]^N$  for $N\to\infty$. These
 Jacobi processes were introduced and studied from different points of views by  Doumerc
 \cite{Do}, Demni \cite{De2, De1},  Remling and R\"osler \cite{RR1, RR2}, and  \cite{V}.
 They depend on 3 parameters and may be described in different ways.
 One possibility, motivated by the theory of  special functions associated with root systems of Heckman and Opdam
 \cite{HS, HO}, is to describe these processes  as time-homogeneous diffusions
 on the alcoves
 $$\tilde A_N:= \{\theta\in [0,\pi]: \> 0\le \theta_N\le \ldots\le \theta_1\le \pi\}$$
 with the Heckman-Opdam Laplacians
$$
  L_{trig, k}f(\theta):= \Delta f(\theta)
 +  \sum_{i=1}^N\Biggl(k_1 \cot(\theta_i/2)+ 2k_2  \cot(\theta_i)\\
  +k_3\sum_{j: j\ne i} \Bigl(\cth(\frac{\theta_i-\theta_j}{2})+\cth(\frac{\theta_i+\theta_j}{2})\Bigr)\Biggr)f_{\theta_i}(\theta).
$$
  of type BC as generators with the  multiplicity parameters
  $k_1,k_2\in \mathbb R$, $k_3>0$ with $k_2\ge0$ and $k_1+k_2\ge 0$  where we assume reflecting boundaries.
It is convenient to transform the processes and their generators in the trigonometric form via the transform
$x_i:=\cos \theta_i$ ($i=1,\ldots,N$) into an algebraic form; see e.g. also  \cite{De2, V}.
We then obtain time-homogeneous diffusions on the alcoves
 $$A_N:=\{x\in\mathbb R^N: \> -1\le x_1\le \ldots\le x_N\le1\}$$
 with the algebraic Heckman-Opdam Laplacians
\begin{equation}\label{generator-algebraic}
 L_kf(x):= \sum_{i=1}^N (1-x_i^2)f_{x_i,x_i}(x)+ \sum_{i=1}^N\Biggl( -k_1-(1+k_1+2k_2)x_i
+2k_3\sum_{j: j\ne i} \frac{1-x_i^2}{x_i-x_j}\Biggr)f_{x_i}(x).
\end{equation}
 as generators with the  multiplicities $k_1,k_2,k_3$ with reflecting boundaries.
The eigenfunctions of the  $L_k$ may be described via Heckman-Opdam Jacobi polynomials,
and the transition probabilities of the Jacobi processes can
be expressed via series expansions in terms of these polynomials; see \cite{De2, RR1, RR2}.
On the other hand, these processes $(X_t=(X_{t,1}, \ldots,X_{t,N}))_{t\ge0}$
admit a description as a unique strong solution of the stochastic differential 
equation (SDE)
 \begin{equation}\label{SDE-alcove-k}
dX_{t,i} = \sqrt{2(1-X_{t,i}^2)}\> dB_{t,i} + \Bigl( -k_1-(1+k_1+2k_2)X_{t,i} +2k_3\sum_{j: j\ne i} \frac{1-X_{t,i}^2}{X_{t,i}-X_{t,j}}\Bigr)dt
 \end{equation}
for $i=1,\ldots,N$
with an $N$-dimensional Brownian motion $(B_t)_{t\ge0}$.
The paths of  $(X_t)_{t\ge0}$ are reflected on $\partial A_N$
and  we start in some point in the interior of $A_N$; see 
Theorem 2.1 of \cite{De2}. It is also possible to start the processes satisfying these SDEs on the boundary.
This is not shown precisely for this type of Jacobi processes on compact alcoves in the literature, but it may be shown in a similar way as for
multivariate Bessel  and  Jacobi processes on non-compact domains in \cite{GM, Sch1, Sch2}.

Following \cite{De2}, we introduce the parameters
\begin{equation}\label{parameter-change-k-p}
\kappa:=k_3>0, \quad q:= N-1+\frac{1+2k_1+2k_2}{2k_3}>N-1, \quad p:= N-1+\frac{1+2k_2}{2k_3}>N-1,
\end{equation}
and
rewrite (\ref{SDE-alcove-k}) as
\begin{align}\label{SDE-alcove}
dX_{t,i} &= \sqrt{2(1-X_{t,i}^2)}\> dB_{t,i} +\kappa\Bigl((p-q) +(2(N-1)-(p+q))X_{t,i}\notag\\
& \quad\quad\quad\quad\quad\quad \quad\quad\quad\quad\quad +
2\sum_{j: \>j\ne i}\frac{1-X_{t,i}^2}{X_{t,i}-X_{t,j}}\Bigr)dt\\
 &=
\sqrt{2(1-X_{t,i}^2)}\> dB_{t,i} +\kappa\Bigl((p-q) -(p+q)X_{t,i} +
2\sum_{j: \> j\ne i}\frac{1-X_{t,i}X_{t,j}}{X_{t,i}-X_{t,j}}\Bigr)dt
\notag
\end{align}
for $i=1,\ldots,N$ and $t>0$.
It is  known
 (see e.g.~\cite{De2, Do}) that  for
 $\kappa\ge 1$ and 
$p,q\ge N-1+2/\kappa$,
 the process does not meet  $\partial A_N^A$  almost surely.
 
It is useful, also to consider the transformed processes
$(\tilde X_{t}:=X_{t/\kappa})_{t\ge0}$ which satisfy
\begin{equation}\label{SDE-alcove-normalized}
d\tilde X_{t,i} = 
\frac{\sqrt 2}{\sqrt\kappa } \sqrt{1-\tilde X_{t,i}^2}\> d\tilde B_{t,i} +\Bigl((p-q) -(p+q)\tilde X_{t,i} +
2\sum_{j:\> j\ne i}\frac{1-\tilde X_{t,i}\tilde X_{t,j}}{\tilde X_{t,i}-\tilde X_{t,j}}\Bigr)dt 
\end{equation}
for $  i=1,\ldots,N$. For $\kappa=\infty$ and $p,q>N-1$, these SDEs with start in $x_0\in A_N$ degenerate to the ODE 
\begin{equation}\label{ODE_main}
	\begin{gathered}
		\frac{d}{dt}x_i(t)
		=(p-q)-(p+q)x_i(t)
			+2\sum_{j:\> j\neq i}\frac{1-x_i(t)
			x_j(t)}{x_i(t)-x_j(t)}\,,\quad i=1,\dots,N\,,\;t>0\,,
			\\
		x(0)=x_0\,.
	\end{gathered}
	\end{equation}
This ODE is interesting for itself and is closely related to the classical  one-dimensional Jacobi polynomials
$(P_N^{(\alpha,\beta)})_{N\ge 0}$ on $[-1,1]$ with the parameters
$$\alpha:=q-N>-1, \quad \beta:=p-N>-1.$$
These polynomials are orthogonal 
 w.r.t.~the weights $(1-x)^\alpha(1+x)^\beta$   on $]-1,1[$ as usual; see Ch.~4 of  \cite{S}.
    All essential informations about   \eqref{ODE_main} are collected in the following theorem which will
    be proved in the appendix in Section 6:

  \begin{theorem}\label{main-solutions-ode}
          Let $N\in\mathbb{N}$ and $p,q> N-1$. Then for each
		 each $x_0\in A_N$ the ODE \eqref{ODE_main} has a unique 
		 solution $x(t)$ for all $t\geq0$ in the following sense:
                 If $x_0$ is in the interior of $A_N$, then  $x(t)$ exists  also in the interior of $A_N$ for all $t\geq0$.
                 Moreover, for $x_0\in\partial A_N$, 
		there is a unique continuous function $x:[0,\infty)\to A_N$ with $x(0)=x_0$ and
		  $x(t)$ being in the interior of $A_N$ for $t>0$ such that $x(t)$ satisfies \eqref{ODE_main} for $t>0$.
                  
		Furthermore, for each  $x_0\in A_N$, the solution satisfies
		$\lim_{t\to\infty}x(t)=z$ where the coordinates of the vector $z$ in the interior of $A_N$ are
                the  ordered roots of the Jacobi polynomial $P_N^{(q-N,p-N)}$.
                This vector $z$ is the only stationary solution of  \eqref{ODE_main} in $A_N$.
	\end{theorem}
	
  The stationary solution $z\in A_N$ in the deterministic case is the freezing limit for $\kappa\to\infty$ of the stationary
  distributions of the corresponding Jacobi processes with fixed parameters $p,q$; see e.g. \cite{HV} for more details. 
   These stationary distributions are just the distributions
  of the $\beta$-Jacobi
  (or  $\beta$-MANOVA) ensembles on $A_N$ having the Lebesgue densities
  \begin{equation}\label{weight-general}
    c(k_1,k_2,k_3)\cdot \prod_{i=1}^N (1-x_i)^{k_1+k_2-1/2}(1+x_i)^{k_2-1/2} \cdot \prod_{i<j}|x_i-x_j|^{2k_3} 
       \end{equation}
  with known Selberg-type norming constants $c(k_1,k_2,k_3)$.
  We here recapitulate that,  possibly
  after some affine linear-transformation and taking some cosine in all coordinates,
  these probability measures appear as the distributions of the ordered 
eigenvalues of 
the tridiagonal models in \cite{KN, K} and in some  log gas models on $[-1,1]$; see  \cite{F1}.
Moreover, for certain parameters, these distributions and the corresponding Jacobi processes have an interpretation
as invariant distributions and Brownian motions respectively on compact Grassmann manifolds of rank $N$ over the fields
$\mathbb F =\mathbb R, \mathbb C$ and the quaternions by the now classical connection between the Heckman-Opdam theory
and spherical functions; see  \cite{HO, HS, RR1, RR2, De2}. We also point out that, even more generally
for some parameters, these distributions and the corresponding Jacobi processes appear as the ordered eigenvalues of
matrices $B^*B$ for upper left blocks $B$ of size $M\times N$  of Haar distributed random variables and Brownian motions
in the unitary group $U(R,\mathbb F)$ respectively with the dimension parameters $R> M> N$; see \cite{Do, De2} for the details.

We now turn to the main content of this paper. We here follow the approach in \cite{VW1} for Bessel processes and derive several
almost sure limit theorems as $N\to\infty$ for the empirical distributions of
the rescaled Jacobi processes $(\tilde{X}_t^N)_{t\ge0}$ and their deterministic freezing limits for $\kappa=\infty$ which satisfy
the ODE  \eqref{ODE_main}, which are related in their flavour to mean field limits of Serfaty \cite{Se}.
Considering the three involved parameters $p,q,\kappa$, it will turn out that the limits depend on $\kappa$ only in a trivial way
while the parameters $p,q$, namely their dependence on $N$, lead  after suitable affine-linear transformations
to different limiting distributions. The different cases are motivated  by the stationary deterministic case,
where we just have the empirical distributions of
the classical Jacobi polynomials. In this setting several limiting regimes with semicircle,
Marchenko-Pastur and Wachter distributions were derived by Dette and Studden \cite{DS}. We thus follow their decomposition of the cases
and investigate the deterministic case with the ODE  \eqref{ODE_main} first.
For this we derive recurrence relations for the moments as well as PDEs for the  Stieltjes and the R-transforms  of
the empirical distributions of the solutions in a general setting in Section 2 for fixed dimensions;
see in particular Eq.~(\ref{stieltjes-pde}). This PDE can then be applied to the different regimes considered in  \cite{DS}.
We shall do this in Section 3 for two regimes where semicircle and Marchenko-Pastur distributions appear.
In the semicircle case we shall obtain the following result where $\mu_{sc,\tau}$ denotes the semicircle law with support $[-\tau,\tau]$
for $\tau\ge0$ and $\mu_{sc,0}=\delta_0$:

\begin{theorem}\label{semicirc_ODE_thm-intro}
  Consider sequences $(p_N)_{N\in \mathbb N},(q_N)_{N\in \mathbb N}\subset ]0,\infty[$
    with $\lim_{N\to\infty}p_N/N=\infty$ and $\lim_{N\to\infty}q_N/N=\infty$ such that
 $C:=\lim_{N\to\infty}p_N/q_N\geq0$ exists. Define
  $$a_N:=\frac{q_N}{\sqrt{Np_N}}, \quad b_N:=\frac{p_N-q_N}{p_N+q_N} \quad(N\in\mathbb N).$$
    Let $\mu\in M^1(\mathbb {R})$ be a probability measure satisfying
    some moment condition (see Theorem \ref{semicirc_ODE_thm} for the details),
    and  let $(x_N)_{N\in\mathbb {N}}=((x_{1}^N,\dots,x_{N}^N))_{N\in\mathbb {N}}$   be a sequence of starting vectors
$x_N\in A_N$  such that all moments of the empirical measures
	\begin{equation*}
		\mu_{N,0}:=\frac{1}{N}\sum_{i=1}^N\delta_{a_N( x_i^N-b_N)}
	\end{equation*}
        tend  to those of $\mu$ for $N\to\infty$.
Let $x_N(t)$ be the solutions of the ODEs \eqref{ODE_main} with $x_N(0)=x_N$ for $N\in\mathbb N$.
	Then for all $t>0$, all moments of the empirical measures
	\begin{equation*}
		\mu_{N,t/(p_N+q_N)}
		=\frac{1}{N}\sum_{i=1}^N \delta_{a_N( x_i^N(t/(p_N+q_N))-b_N)}\end{equation*}
        tend to those of the probability measures
  $$ \mu_t  :=   (e^{-t}\mu)
    \boxplus\left(\sqrt{1-e^{-2t}}\mu_{sc,4(1+C)^{-3/2}}\right).$$
  This in particular implies that the	$\mu_{N,t/(p_N+q_N)}$ tend weakly to the $\mu_t$.
	\end{theorem}
	
We point out that Theorem \ref{semicirc_ODE_thm-intro} is a local limit theorem on the behaviour of the particle systems
with $N$ particles around the starting points $b_N\in]-1,1[$ for small times on the space scale $1/a_N$ for large $N$.
We also mention the slightly astonishing fact that this  local result preserves the asymptotic stationarity of the global systems.
In fact there are local limit results on different time and space scales in Section 3 where this                   
 asymptotic stationarity does not appear; see e.g.~ Theorem \ref{semicirc_ODE_thm2}. Besides these two results and further
  variants with Wigner-type  limits in Section 3, we shall
 also derive   local limit results with Marchenko-Pastur type limits in neighbourhoods
 of the boundary points $\pm 1$ in Section 3;
see for instance  Theorem \ref{MP_ODE_thm}
below. In the proof of this theorem 
 we again solve
 the associated PDE for the R-transforms explicitly. We point out that a modification of this PDE in the
 Marchenko-Pastur setting appears also in \cite{CG}.

 There are further limit regimes where Kesten-MacKay and Wachter distributions are involved, and which are also motivated by \cite{DS}
 and corresponding  limit results for $\beta$-Jacobi ensembles e.g.~in \cite{DN, W}.
In these cases it can be also shown that under corresponding conditions on the initial conditions, the empirical measures
$\mu_{N,t/(p_n+q_N)}$ also converge to some probability measures $\mu_t$ for $t\ge0$.
However, the details of the description of the limits are more involved here and will be published in the future separately.

The results  of Section 3 on the compact, deterministic case will be extended in Section 4 
to almost sure versions for Jacobi processes with fixed parameter $\kappa$ in the compact setting.
It turns out the limiting distributions stay the same for the rescaled processes.
Hence as for Bessel processes the form of the limiting distribution is already determined by the frozen process.

Furthermore, in Section 5 we  transfer some of our Wigner- and Marchenko-Pastur type results
in the Sections 2-4 to a noncompact setting. For some parameters, these results have interpretations in terms of
Brownian motions on the noncompact Grassmann manifolds over $\mathbb R, \mathbb C$, and the quaternions.
It will turn out that in these hyperbolic cases,
some results remain valid up to some kind of time inversion. However, it seems that here no analogue to the stationary
results like Theorem \ref {semicirc_ODE_thm-intro} are available, as the the initial conditions do not fit 
to the conditions on the parameters $p_N,q_N, a_N, b_N$ in this theorem.
Finally, as mentioned above, we  prove Theorem \ref{main-solutions-ode} and
its noncompact analogue  in Section 6.

\section{Moments of the empirical distributions in the deterministic case}

In this section we study  the  solutions $x^N(t)$ of the ODEs \eqref{ODE_main} for suitable
initial conditions $x_0^N\in A_N$ for
$N\in\mathbb N$
and suitable parameters $p=p_N, q=q_N>N-1$ where we are interested in the case $N\to\infty$ which implies that also
$p=p_N, q=q_N\to\infty$ holds. It will turn out that there are several limit regimes for the empirical measures
$$\frac{1}{N}(\delta_{x^N_1(t)}+\ldots+\delta_{x^N_N(t)})\in M^1([-1,1])$$
for $N\to\infty$ and all $t\ge0$ under the condition that a corresponding limit holds for the initial conditions at time $t=0$.
For some of these limit results we have to transform the data in an affine-linear way in all coordinates depending on $N$.
For this we introduce suitable  sequences $(a_N)_{N\in\mathbb N }\subset]0,\infty[$ and
	$(b_N)_{N\in\mathbb N}\subset\mathbb R $ which will be specified later in several specific situations.
                  We consider the transformed solutions $\tilde{x}^N(t)=(\tilde{x}^N_{i}(t),\ldots, \tilde{x}^N_{N}(t))$ with 
                  $$\tilde{x}^N_{i}(t):=a_N(x^N_{i}(t)-b_N)      \quad(1\le i\le N)$$
                  as well as the transformed empirical distributions
\begin{equation}
  \mu_{N,t}:=\frac{1}{N}(\delta_{\tilde x^N_1(t)}+\ldots+\delta_{\tilde x^N_N(t)})=
  \frac{1}{N}(\delta_{a_N(x_1^N(t)-b_N)}+\ldots+\delta_{a_N(x_N^N(t)-b_N)}).
	\end{equation}
In order to determine possible weak limits of the measures $\mu_{N,t}$, we shall study the moments
\begin{equation}
		S_{N,l}(t)
		:=\int_{[-1,1]}y^l\,d\mu_{N,t}(y)
		=\frac{a_N^{l}}{N}\sum_{i=1}^N(x^N_{i}(t)-b_N)^l
		=\frac{1}{N}\sum_{i=1}^N\tilde{x}^N_{i}(t)^l\,,
	\end{equation}
of these measures for 	$l\in\mathbb N_0$, $t\ge0$, and $N\in\mathbb N$.
In particular we have $S_{N,0}\equiv 1$.
To study higher moments, we rewrite the  ODE \eqref{ODE_main} as an ODE  for $\tilde{x}^N_{i}$ by
	\begin{align}\label{ODE-tilde-x}
		\frac{d}{dt}\tilde{x}_i^N(t)
		=&a_N(p-q-b_N(p+q))-(p+q)\tilde{x}_i^N(t)\\
		&	+2\sum_{j: j\ne i}\frac{a_N^2(1-b_N^2)
			-a_Nb_N(\tilde{x}_i^N(t)+\tilde{x}^N_j(t))
			-\tilde{x}_i^N(t)\tilde{x}^N_j(t)}
			{\tilde{x}_i^N(t)-\tilde{x}^N_j(t)}
	\notag\end{align}
	where we always agree  that a summation
        over $j: j\ne i$ means that we sum over all  $j\ne i$ from $1$ to $N$.
	In the following we  also suppress the dependence of $S_{N,l}$ and $\tilde{x}^N$ on $t$.
 (\ref{ODE-tilde-x}) yields the following  ODEs for the $S_{N,l}$ for  $l\in\mathbb N$:
	\begin{align}
		\frac{d}{dt}S_{N,l}
		&=\frac{l }{N}\sum_{i=1}^n\left(\tilde{x}_i^N\right)^{l-1}
			\left(\frac{d}{dt}\tilde{x}_i^N\right)\\		
		        &=\frac{l}{N}
\Biggl[ a_N(p-q-b_N(p+q))N \cdot  S_{N,l-1} -(p+q)N\cdot S_{N,l} \notag\\
			&\quad\quad
  +2\sum_{i,j: \> i\neq j} \frac{(a_N^2(1-b_N^2)-\tilde{x}_i^N\tilde{x}^N_j)\left(\tilde{x}_i^N\right)^{l-1}  - b_Na_N(\left(\tilde{x}_i^N\right)^l+\left(\tilde{x}_i^N\right)^{l-1}\tilde{x}^N_j)}{\tilde{x}_i^N-\tilde{x}^N_j}\Biggr]\,.
	\notag
	\end{align}
        In particular, for $l=1$,
     \begin{equation}   \frac{d}{dt}S_{N,1}=a_N(p-q-b_N(p+q))-(p+q)S_{N,1}.\end{equation}
%     Moreover, for $l=2$,
%	\begin{align}
%\frac{d}{dt}S_{N,2}
%		=&2\Biggl[
%		  -(p+q)S_{N,2}+a_N(p-q-b_N(p+q))S_{N,1}\\
%                  &\quad +
%                  \frac{2}{N}\sum_{i,j: \> i\neq j} \frac{(a_N^2(1-b_N^2)-
%                    \tilde{x}_i^N\tilde{x}^N_j)\tilde{x}_i^N - b_Na_N(\left(\tilde{x}_i^N\right)^2+\tilde{x}_i^N\tilde{x}^N_j)}{\tilde{x}_i^N-\tilde{x}^N_j}\Biggr]\notag\\
%		=&2\Bigr[
%                  -(p+q)S_{N,2}+a_N(p-q-b_N(p+q))S_{N,1}\notag\\
%&\quad+\frac{1}{N}           \sum_{i,j: \> i\neq j}\Bigl(a_N^2(1-b_N^2)-\tilde{x}_i^N\tilde{x}_j^N
%		-b_Na_N(\tilde{x}_i^N+\tilde{x}_j^N)\Bigr)\Bigr]\notag\\
%		=&2\Bigl[-(p+q)S_{N,2}+a_N(p-q-b_N(p+q))S_{N,1}+a_N^2(1-b_N^2)(N-1)\notag\\
%                	&\quad  -NS_{N,1}^2
%		+S_{N,2}-2b_Na_N(N-1)S_{N,1}\Bigr]\notag\\
%		=&2\Bigl[-(p+q-1)S_{N,2}+a_N(p-q-b_N(p+q)-2(N-1)b_N)S_{N,1}\notag\\
%                 &\quad -NS_{N,1}^2+a_N^2(1-b_N^2)(N-1)\Bigr]\,.\notag
%	\end{align}
%	Furthermore, for $l\geq3$ we first observe that
	Moreover, for $l\geq2$ we first observe that
	\begin{align}
		2&\sum_{i,j: \> i\neq j}(a_N^2(1-b_N^2)-\tilde{x}_i^N\tilde{x}_j^N)
				\frac{\left(\tilde{x}_i^N\right)^{l-1}}{\tilde{x}_i^N-\tilde{x}_j^N}
		=2\sum_{i,j: \> i< j}(a_N^2(1-b_N^2)-\tilde{x}_i^N\tilde{x}_j^N)
				\frac{\left(\tilde{x}_i^N\right)^{l-1}-\left(\tilde{x}_j^N\right)^{l-1}}{\tilde{x}_i^N-\tilde{x}_j^N}\notag\\
		&=\sum_{k=0}^{l-2}\sum_{i,j: \> i\neq j}(a_N^2(1-b_N^2)-\tilde{x}_i^N\tilde{x}_j^N)\left(\tilde{x}_i^N\right)^k\left(\tilde{x}^N_j\right)^{l-2-k}\notag\\
		&=a_N^2(1-b_N^2)\Biggl(N^2\sum_{k=0}^{l-2}S_{N,k}S_{N,l-2-k}-(l-1)NS_{N,l-2}\Biggr)\notag\\
		&\quad\quad-N^2\sum_{k=0}^{l-2}S_{N,k+1}S_{N,l-1-k}+(l-1)NS_{N,l}\,.
\notag	\end{align}
	Furthermore, with the usual convention for empty sums, 
	\begin{align}
	2&\sum_{i,j: \> i\neq j}\frac{\left(\tilde{x}_i^N\right)^l+\left(\tilde{x}_i^N\right)^{l-1}\tilde{x}_j^N}
				{\tilde{x}_i^N-\tilde{x}_j^N}=
	2\sum_{i,j=1: \> i<j}\frac{\left(\tilde{x}_i^N\right)^l-\left(\tilde{x}^N_j\right)^l}{\tilde{x}_i^N-\tilde{x}_j^N}
		+2\sum_{i,j: \> i<j}\tilde{x}_i^N\tilde{x}_j^N\frac{\left(\tilde{x}_i^N\right)^{l-2}
		-\left(\tilde{x}_j^N\right)^{l-2}}
			{\tilde{x}_i^N-\tilde{x}_j^N}\notag\\
	=&\sum_{k=0}^{l-1}\sum_{i,j: \> i\neq j}\left(\tilde{x}_i^N\right)^k\left(\tilde{x}_j^N\right)^{l-1-k}
		+\sum_{k=0}^{l-3}\sum_{i,j: \> i\neq j}\left(\tilde{x}_i^N\right)^{k+1}\left(\tilde{x}_j^N\right)^{l-2-k}\notag\\
	=&N^2\sum_{k=0}^{l-1}S_{N,k}S_{N,l-1-k}-NlS_{N,l-1} +N^2\sum_{k=0}^{l-3}S_{N,k+1}S_{N,l-2-k}-N(l-2)S_{N,l-1}\notag\\
	=&N^2\sum_{k=0}^{l-1}S_{N,k}S_{N,l-1-k} +N^2\sum_{k=0}^{l-3}S_{N,k+1}S_{N,l-2-k}-2N(l-1)S_{N,l-1}\notag\\
	=&N^2\left[\sum_{k=0}^{l-2}S_{N,k}S_{N,l-1-k}+S_{N,l-1}
		+\sum_{k=0}^{l-2}S_{N,k+1}S_{N,l-2-k}-S_{N,l-1}\right]-2N(l-1)S_{N,l-1}\notag\\
	=&2N^2\sum_{k=0}^{l-2}S_{N,k}S_{N,l-1-k}-2N(l-1)S_{N,l-1}\,.
\notag	\end{align}
Therefore, for $l\geq2$, and $p=p_N,q=q_N$,
	\begin{align}\label{moment-ODE1}
		\frac{d}{dt}S_{N,l}
		=l&\Bigl[(p-q-b_N(p+q-2(l-1)))a_N S_{N,l-1}-(p+q-(l-1))S_{N,l}\notag\\
                  &-a_N^2(1-b_N^2)(l-1)S_{N,l-2}
		+Na_N^2(1-b_N^2)\sum_{k=0}^{l-2}S_{N,k}S_{N,l-2-k}\notag\\&-N\sum_{k=0}^{l-2}S_{N,k+1}S_{N,l-1-k}
		-2a_Nb_N N\sum_{k=0}^{l-2}S_{N,k}S_{N,l-1-k}\Bigr]\,.
                	\end{align}
In summary,  we have the  recursion (\ref{moment-ODE1}) together with
        \begin{equation}\label{moment-ODE2}
		\frac{d}{dt}S_{N,0}\equiv 0,\quad\quad
		\frac{d}{dt}S_{N,1}=-(p+q)S_{N,1}+a_N(p-q-b_N(p+q)).
	        	\end{equation}
In the next step we consider the Cauchy transforms of the measures $\mu_{N,t}$. For this we recapitulate that for
        $\mu\in M^1(\mathbb R)$ the Cauchy transform is given by
$$G_{\mu}(z):=\int_{\mathbb R}\frac{1}{z-x}\,d\mu(x) \quad\quad
	(z\in\{z\in\mathbb C\colon\,\Im(z)>0\}).$$
	 We set
	$G^N(t,z):=G_{\mu_{N,t}}(z)$. For $\lvert z\rvert$ sufficiently large we can write
	$G^N$ as
	\begin{equation}\label{series-cauchy}
		G^N(t,z)=\sum_{l=0}^{\infty}z^{-(l+1)}S_{N,l}(t)\,.
	\end{equation} 
	We now consider the partial derivatives  $G_t^N(t,z):=\partial_t G^N(t,z)$ and
	$G_z^N(t,z):=\partial_z G^N(t,z)$ and similarly for higher orders.
	(\ref{series-cauchy}) thus  leads to
	\begin{equation}\label{series-cauchy2}
		G^N_t(t,z)=\sum_{l=0}^{\infty}z^{-(l+1)}\frac{d}{dt}S_{N,l}(t)
		=\sum_{l=1}^{\infty}z^{-(l+1)}\frac{d}{dt}S_{N,l}(t)\,.
	\end{equation}
	We now  calculate this series by using  (\ref{moment-ODE1}) and (\ref{moment-ODE2}). For this we
use the following equations:
	\begin{align}\label{partial1}
		-\sum_{l=1}^{\infty}z^{-(l+1)}l(p+q-(l-1))S_{N,l}=&-(p+q)\sum_{l=1}^{\infty}z^{-(l+1)}lS_{N,l}+\sum_{l=1}^{\infty}z^{-(l+1)}l(l-1)S_{N,l}\\
		=&(p+q)zG^N_z(t,z)+(p+q)G^N(t,z)+\partial_{zz}\left(z^2G^N(t,z)\right)\,,\notag
                \end{align}
	\begin{align}\label{partial2}
		\sum_{l=1}^{\infty}lz^{-(l+1)}a_N(p-q-&b_N((p+q)-2(l-1)))S_{N,l-1}\\=&-a_N(p-q-b_N(p+q))G^N_z(t,z)+2a_Nb_N\partial_{zz}\left(zG^N(t,z)\right)\,,\notag
                \end{align}
	\begin{equation}\label{partial3}
		-\sum_{l=2}^{\infty}z^{-(l+1)}l(l-1)S_{N,l-2}=-G_{zz}^N(t,z), \quad
		\sum_{l=2}^{\infty}z^{-(l+1)}lN\sum_{k=0}^{l-2}S_{N,k}S_{N,l-2-k}=-2NG^N(t,z)G^N_z(t,z)\,,
	\end{equation}
	\begin{equation}\label{partial5}
		-\sum_{l=2}^{\infty}z^{-(l+1)}lN\sum_{k=0}^{l-2}S_{N,k+1}S_{N,l-1-k}=2N(z^2G^N(t,z)G^N_z(t,z)+z(G^N(t,z))^2-zG^N_z(t,z)-G^N(t,z))\,,
	\end{equation}
	and
	\begin{align}\label{partial6}
		-\sum_{l=2}^{\infty}z^{-(l+1)}l\sum_{k=0}^{l-2}&S_{N,k}S_{N,l-1-k}
		=\partial_z\left[\sum_{l=2}^{\infty}z^{-l}\sum_{k=0}^{l-2}S_{N,k}S_{N,l-1-k}\right]\notag\\
		=&\partial_z\left[\sum_{l=1}^{\infty}z^{-(l+1)}\sum_{k=0}^{l-1}S_{N,k}S_{N,l-k}\right]\notag\\
		=&\partial_z\left[\sum_{l=1}^{\infty}z^{-(l+1)}\sum_{k=0}^{l}S_{N,k}S_{N,l-k}-\sum_{l=1}^{\infty}z^{-(l+1)}S_{N,l}\right]\notag\\
		=&\partial_z\left[\sum_{l=0}^{\infty}z^{-(l+1)}\sum_{k=0}^lS_{N,k}S_{N,l-k}-z^{-1}-G^N+z^{-1}\right]\notag\\
		=&\partial_z\left[z(G^N)^2-G^N\right]
		=(G^N)^2+2zG_z^NG^N-G_z^N\,.
	\end{align}
	If we combine (\ref{partial1})-(\ref{partial6}) with (\ref{moment-ODE1}), (\ref{moment-ODE2}),  and  (\ref{series-cauchy2}), we finally obtain 
the  PDE
	\begin{align}\label{stieltjes-pde}
		G_t^N&(t,z)\\
		=&(p+q)zG^N_z(t,z)+(p+q)G^N(t,z)+\partial_{zz}\left(z^2G^N(t,z)\right)-a(p-q-b(p+q))G^N_z(t,z)\notag\\
			&+2ab\partial_{zz}\left(zG^N(t,z)\right)-(1-b^2)a^2G_{zz}^N(t,z)-2Na^2(1-b^2)G^N(t,z)G^N_z(t,z)\notag\\
		&+2N\left[z^2G^N(t,z)G^N_z(t,z)+z(G^N(t,z))^2-zG^N_z(t,z)-G^N(t,z)\right]\notag\\
		&+2bNa((G^N)^2+2zG_z^NG^N-G_z^N)\,\notag
	\end{align}
        for the Cauchy transforms $G^N(t,z)$ of the  measures $\mu_{N,t}$.
        This PDE can be used  to derive limit theorems for the $\mu_{N,t}$ under different assumptions on the parameters
        $p=p_N,q=q_N, a_N, b_N$
        for $N\to\infty$ and $t\ge0$. We present  such limit results in the next section where in the limit roughly
        free sums of the limit initial distributions with
        Wigner- and Marchenko-Pastur distributions appear.

\section{Wigner- and Marchenko-Pastur-type limit theorems in the deterministic case}

In this section we study several conditions  on the parameters $p_N,q_N, a_N, b_N$ above leading to
limit results for the  measures  $\mu_{N,t}$ which involve semicircle and Marchenko-Pastur distributions.
In both cases, we  consider $a_N\to\infty$ which implies that we must work possibly with measures with noncompact supports.
We thus need some condition on the moments.
We recapitulate e.g.~from \cite{A} that a probability measure $\mu\in M^1(\mathbb {R})$ satisfies the
	{\it Carleman condition} if  the moments $c_l=\int_{\mathbb {R}}x^l\,d\mu(x)$
	($l\in\mathbb {N}$), of $\mu$ satisfy
	\begin{equation}\label{Carleman}
		\sum_{l=1}^{\infty}c_{2l}^{-\frac{1}{2l}}=\infty\,.
	\end{equation}
 By  \cite{A},  a probability measure with the Carleman condition is determined uniquely by its moments.

We also recapitulate the  R-transform  of $\mu\in M^1(\mathbb {R})$ from \cite{AGZ}, which is given by
        	$R_{\mu}(z):=\sum_{n=0}^{\infty}k_{n+1}(\mu)z^n$ with the  $n$-th free
cumulants $k_n(\mu)$ of $\mu$. It is related to the Cauchy transform by
\begin{equation}\label{R-cauchy-connection}
  R_{\mu}(G_{\mu}(z))=z-1/G_{\mu}(z).
\end{equation}
Furthermore, the  R-transform satisfies
 $R_{\mu\boxplus\nu}=R_{\mu}+R_{\nu}$ for  $\mu,\nu\in M^1(\mathbb {R})$ for
 the free additive convolution $\boxplus$.

We shall also use the following notation:
 We denote the image of some probability measure 
$\mu\in M^1(\mathbb R)$ under some continuous mapping $f$ by  $f(\mu)$. We use this notation in particular for the maps
 $x\mapsto|x|$ and  $x\mapsto x^2$ and write $|\mu|$ and $\mu^2$. Moreover, for a constant  $v\in\mathbb {R}\setminus\{0\}$
 let 	$v\mu$ the image of $\mu$ under the map $x\mapsto vx$.
Finally, for a probability measure $\mu$ on $[0,\infty[$, let 
 $\mu_{even}$ the unique even  probability measure on $\mathbb R$ with $|\mu_{even}|=\mu$.

    With these notations we have
	$G_{v\mu}(z)=v^{-1}G_{\mu}\left(z/v\right)$ and thus, by (\ref{R-cauchy-connection}),
	\begin{equation}\label{R-scaling}
		R_{v\mu}(z)=vR_{\mu}(vz).
	\end{equation}

        We now turn to the first limit case where semicircle laws $\mu_{sc,\lambda}\in M^1(\mathbb R)$ with radius $\lambda>0$
        appear. We recapitulate that the Wigner law $\mu_{sc,\lambda}$ with radius $\lambda>0$
 	has the Lebesgue  density
	$$\frac{2}{\pi \lambda^2}\sqrt{\lambda^2-x^2}{\bf 1}_{[-\lambda,\lambda]}(x).$$
        It is well-known  that $R_{\mu_{sc,\lambda}}(z)=\frac{\lambda^2}{4}z$; see Section 5.3 of \cite{AGZ}.
We have the following first result:

\begin{theorem}\label{semicirc_ODE_thm}
  Consider sequences $(p_N)_{N\in \mathbb N},(q_N)_{N\in \mathbb N}\subset ]0,\infty[$
    with $\lim_{N\to\infty}p_N/N=\infty$ and $\lim_{N\to\infty}q_N/N=\infty$ such that
 $C:=\lim_{N\to\infty}p_N/q_N\geq0$ exists. Define
  $$a_N:=\frac{q_N}{\sqrt{Np_N}}, \quad b_N:=\frac{p_N-q_N}{p_N+q_N} \quad(N\in\mathbb N).$$
Let $\mu\in M^1(\mathbb {R})$ be a probability measure such that its  moments $c_l$  satisfy
	$\lvert c_l\rvert\leq(\gamma l)^l$ for $l\in\mathbb {N}_0$ with some constant $\gamma>0$.
  Moreover, let $(x_N)_{N\in\mathbb {N}}=((x_{1}^N,\dots,x_{N}^N))_{N\in\mathbb {N}}$   be a sequence of starting vectors
$x_N\in A_N$  such that all moments of the empirical measures
	\begin{equation*}
		\mu_{N,0}:=\frac{1}{N}\sum_{i=1}^N\delta_{a_N( x_i^N-b_N)}
	\end{equation*}
        tend  to those of $\mu$ for $N\to\infty$.
Let $x_N(t)$ be the solutions of the ODEs \eqref{ODE_main} with $x_N(0)=x_N$ for $N\in\mathbb N$.
	Then for all $t>0$, all moments of the empirical measures
	\begin{equation*}
		\mu_{N,t/(p_N+q_N)}
		=\frac{1}{N}\sum_{i=1}^N \delta_{a_N( x_i^N(t/(p_N+q_N))-b_N)}\end{equation*}
        tend to those of the probability measures
  \begin{equation}\label{limit-semi-main}
  	\mu_t  :=   (e^{-t}\mu)
			\boxplus\left(\sqrt{1-e^{-2t}}\mu_{sc,4(1+C)^{-3/2}}\right).
	\end{equation}
\end{theorem}
	
\begin{proof}
  Using the recurrence relations \eqref{moment-ODE1}, \eqref{moment-ODE2} together with the initial conditions for $t=0$ and our choice of $b_N$,
  we see that the moments $\tilde S_{N,l}(t):=S_{N,l}(t/(p_N+q_N))$ of
  $\mu_{N,t/(p_N+q_N)}$ satisfy
 $$\tilde S_{N,0}\equiv 1, \quad
	        \tilde S_{N,1}(t)=e^{-t}S_{N,1}(0)$$
                and, for $l\ge 2$,
                \begin{align}\label{moment-ODE-wigner}	\tilde	S_{N,l}(t)&=
                  \exp\Bigl(\Bigl(-l+\frac{l(l-1)}{p_N+q_N}\Bigr)t\Bigr)\Biggl[S_{N,l}(0)\\
        &\quad+\frac{l}{p_N+q_N}\int_0^t \exp\Bigl(\Bigl(l-\frac{l(l-1)}{p_N+q_N}\Bigr)s\Bigr)\Biggl(
2a_Nb_N(l-1) \tilde S_{N,l-1}(s)\notag\\&\quad\quad\quad\quad\quad -(1-b_N^2)a_N^2(l-1) \tilde S_{N,l-2}(s)
 +Na_N^2(1-b_N^2)\sum_{k=0}^{l-2} \tilde S_{N,k}(s) \tilde S_{N,l-2-k}(s)\notag\\
&\quad\quad\quad\quad\quad-
N\sum_{k=0}^{l-2} \tilde S_{N,k+1}(s) \tilde S_{N,l-1-k}(s)
-2b_NNa_N\sum_{k=0}^{l-2} \tilde S_{N,k}(s) \tilde S_{N,l-1-k}(s)
\Biggr)ds\Biggr].
	        \notag	\end{align}
As the starting moments $S_{N,l}(0)$ ($l\ge0$) converge to the corresponding moments of $\mu$
for $N\to\infty$, we conclude by induction on $l$, that the 
$\tilde S_{N,l}(t)$ converge  to some  functions $S_l(t)$ for  $l\ge0$ and $t\ge0$. Moreover, these limits satisfy
\begin{equation}\label{recurrence-wigner-limit}
  S_0\equiv1,\quad
		S_1(t)=S_1(0)e^{-t}\,,\quad
		S_l(t)=e^{-lt}\left(S_l(0)+4l(1+C)^{-3}\int_0^te^{ls}\sum_{k=0}^{l-2}S_k(s)S_{l-2-k}(s)
		\,ds\right)
                \end{equation}
for $l\ge2$. We will now prove that the $S_l(t)$ satisfy the Carleman condition \eqref{Carleman} for $t>0$ so that,
	by the moment convergence theorem, there exist unique $\mu_t\in M^1(\mathbb {R})$ with
	$(S_l(t))_l$ as  sequences of moments. For this we show that there exists
	an $R>1$ such that $\lvert S_l(t)\rvert\leq (Rl)^l$ for all $t\ge0$ and
	$l\in\mathbb {N}_0$. Clearly this holds for $l\in\{0,1\}$ for $R$ sufficiently large.
	Moreover, by induction we have for $l\geq2$ and $t\ge0$ that
	\begin{equation}\label{ode_semicirc_pf_eq1}
	\begin{split}
		\lvert S_l(t)\rvert
		&\leq e^{-lt}\lvert S_l(0)\rvert+ e^{-lt}4l(1+C)^{-3}\int_0^t e^{ls}
		\sum_{k=0}^{l-2}\lvert S_k(s)\rvert\,\lvert S_{l-2-k}(s)\rvert\,ds\\
	&= e^{-lt}\lvert S_l(0)\rvert+ 4l(1+C)^{-3}\int_0^t e^{-ls}
		\sum_{k=0}^{l-2}\lvert S_k(t-s)\rvert\,\lvert S_{l-2-k}(t-s)\rvert\,ds\\
    &\leq (\gamma l)^l+4(1+C)^{-3}(Rl)^{l-2}
		\leq (\gamma l)^l+R^{l-2}l^l\,.
	\end{split}
	\end{equation}
	For $R$ large enough (depending on $\gamma$) we can bound the RHS of
	\eqref{ode_semicirc_pf_eq1} by $(Rl)^l$ as claimed. We thus see that
	$(S_l(t))_{l\in\mathbb {N}_0}$ satisfies the Carleman condition for  $t\geq0$.
        We thus conclude that the measures $\mu_{N,t/(p_N+q_N)}$ tend weakly to  some probability measures $\mu_t$.
        
	To identify the $\mu_t$ we employ a PDE for the corresponding Cauchy and R-transforms.
	We set
	\begin{equation*}
		G(t,z):=G_{\mu_t}(z)
		=\lim_{N\to\infty}G_{\mu_{N,t/(p_n+q_N)}}(z)\,.
	\end{equation*}
	We now use the PDEs \eqref{stieltjes-pde} and interchange derivatives w.r.t.~$t,z$ with the limits
        $N\to\infty$. This interchangeability can be proved via the Laurent series for $G,G_N$ as in Proposition 2.9 of \cite{VW1}.
        In this way we obtain
       that  $G$ satisfies the PDE
	\begin{equation*}
		G_t(t,z)=zG_z(t,z)+G(t,z)-8(1+C)^{-3}G(t,z)G_z(t,z)\,,\quad G(0,z)=G_{\mu}(z)\,.
	\end{equation*}
Using the	transformation rules
	\begin{align}\label{R-Cauchy}
		R(t,G(t,z))&=z-1/G(t,z)\\
		R_z(t,G(t,z))&=1/G_z(t,z)+1/G^2(t,z)\notag\\
		R_t(t,G(t,z))&=-G_t(t,z)/G_z(t,z)\,.\notag
	\end{align}
for the R-transforms $R(t,z):=R_{\mu_t}(z)$, we see that
	\begin{equation}\label{semicirc_R_PDE}
		R_t(t,z)=-R(t,z)+8(1+C)^{-3}z-R_z(t,z)z\,,\quad R(0,z)=R_{\mu}(z)\,.
	\end{equation}
	As the solution of \eqref{semicirc_R_PDE} is given by
	\begin{equation*}
		R(t,z)=e^{-t}R_{\mu}(ze^{-t})+4(1+C)^{-3}(1-e^{-2t})z\,,
	\end{equation*}
	it follows from  \eqref{R-scaling} and the further properties of
the R-transform mentioned above  that
	\begin{equation*}
		\mu_t
		=(e^{-t}\mu)\boxplus\left(\sqrt{1-e^{-2t}}
			\mu_{sc,4(1+C)^{-3/2}}\right)\,
	\end{equation*}
        as claimed.
\end{proof}

\begin{remark}
 The exchange of the  $p_N,q_N$ in our dynamical systems corresponds to a sign change (and thus a reverse numbering)
   of all particles in $[-1,1]$. In this way we may assume w.l.o.g.~that $C:=\lim_{N\to\infty}p_N/q_N\in [0,1]$
   holds in Theorem \ref{semicirc_ODE_thm}.
   Moreover, the degenerated case $C=\infty$ corresponds to the  degenerated case $C=0$ and is thus also included in
   Theorem \ref{semicirc_ODE_thm} in principle.
\end{remark}

In order to understand the meaning of Theorem \ref{semicirc_ODE_thm}, consider the following example:

\begin{example} Let $p_N,q_N, a_N, b_N$ be given as in Theorem \ref{semicirc_ODE_thm}, and take $x_i^N:=b_N$ for all $i,N$.
  Then all $\mu_{N,0}=\delta_0$ and $\mu_0=\delta_0$. In this case the measures $\mu_t$ from (\ref{limit-semi-main}) are the semicircle laws
  $\mu_t=\sqrt{1-e^{-2t}}\mu_{sc,4(1+C)^{-3/2}}$ for $t>0$. These measures describe the deviation of the particles 
  $x_i^N(t)$ at time $t/(p_N+q_N)$ from the numbers $b_N\in ]-1,1[$ locally w.r.t.~to the space scalings $a_N$.

  Notice that this even makes sense for the degenerated case $C=0$ where $\lim_{N\to\infty}b_N=-1$ holds.  
\end{example}

In summary,  Theorem \ref{semicirc_ODE_thm} 
is a local limit theorem which  describes the behaviour of the system  around the numbers $b_N$ for small times.
It is therefore astonishing that in the limit (\ref{limit-semi-main})
a stationary behaviour appears which is available on the global scale  of the particle processes on $[-1,1]$.
This picture appears also
in a variant of Theorem \ref{semicirc_ODE_thm} in the degenerated case $C=0$
in the following Theorem \ref{semicirc_ODE_thm2}. However,
   this stationarity disappears if we use  scalings  in space and 
   time of higher orders  than in Theorem \ref{semicirc_ODE_thm}; see Theorem \ref{semicirc_ODE_thm3} below.

\begin{theorem}\label{semicirc_ODE_thm2}
  Consider sequences $(p_N)_{N\in \mathbb N},(q_N)_{N\in \mathbb N}\subset ]0,\infty[$
    with $\lim_{N\to\infty}p_N/N=\infty$ and $\lim_{N\to\infty}q_N/N=\infty$ and
 $C:=\lim_{N\to\infty}p_N/q_N=0$. Define
  $$a_N:=\frac{\sqrt{q_N}}{\sqrt{N}}, \quad b_N:=\frac{p_N-q_N}{p_N+q_N} \quad(N\in\mathbb N).$$
    Let $\mu\in M^1(\mathbb {R})$ be a starting  measure and  $(x_N)_{N\in\mathbb {N}}$
     starting vectors
$x_N\in A_N$ as in Theorem \ref{semicirc_ODE_thm}.
Let $x_N(t)$ be the solutions of the ODEs \eqref{ODE_main} with $x_N(0)=x_N$ for $N\in\mathbb N$.
Then for all $t>0$, all moments of the  measures	$\mu_{N,t/(p_N+q_N)}$  as in Theorem \ref{semicirc_ODE_thm}
 tend to those of the  measure $e^{-t}\mu$.
\end{theorem}

\begin{proof}
  The proof is completely analogue to that of Theorem \ref{semicirc_ODE_thm}. We thus skip the proof. We only point out that the limit can be
  interpreted as  $(e^{-t}\mu)\boxplus(\sqrt{1-e^{-2t}}\mu_{sc,0})$ where the  semicircle law degenerates into $\mu_{sc,0}=\delta_0$.
\end{proof}

We next consider a further variant of  Theorem \ref{semicirc_ODE_thm} with a different scaling  in space and 
time  where the limit loses its stationary behaviour, and where
the limit corresponds to the results for the Bessel processes of type A and their frozen versions
in Sections 2 and 3 of \cite{VW1}. We point out that here the conditions on the parameters $p_N,q_N, b_N$
are much more flexible,
and that this result admits an analogue for Jacobi processes on noncompact spaces; see Section 5.

\begin{theorem}\label{semicirc_ODE_thm3}
  Consider sequences $(p_N)_{N\in \mathbb N},(q_N)_{N\in \mathbb N}\subset ]0,\infty[$
    with $p_N,q_N>N-1$ for $N\ge 1$.  Let $(b_N)_{N\in \mathbb N}\subset]-1,1[$ be any sequence such that 
    $B:=\lim b_N\in[-1,1]$ exists.
 Let  $(s_N)_{N\in \mathbb N}\subset]0,\infty[$ be a sequence of time scalings  with
    $$\lim_{N\to\infty}\frac{p_N+q_N}{\sqrt{Ns_N}}=0,$$
    and define the space scalings
$ a_N:=\sqrt{s_N/N}$.

       Let $\mu\in M^1(\mathbb {R})$ be a starting  measure and  $(x_N)_{N\in\mathbb {N}}$
     starting vectors as in Theorem \ref{semicirc_ODE_thm}.
Let $x_N(t)$ be the solutions of the ODEs \eqref{ODE_main} with $x_N(0)=x_N$ for $N\in\mathbb N$.
Then for all $t>0$, all moments of the empirical measures
\begin{equation*}
	\mu_{N,t/s_N}=\frac{1}{N}\sum_{i=1}^N \delta_{a_N( x_i^N(t/s_N)-b_N)}
\end{equation*}
 tend to those of  $   \mu\boxplus\mu_{sc,2\sqrt{2(1-B^2)t}}$.
\end{theorem}

\begin{proof}
  The proof is again  analog to that of Theorem \ref{semicirc_ODE_thm}. In fact, 
  the recurrence relations \eqref{moment-ODE1}, \eqref{moment-ODE2} show
   that here the moments $\tilde S_{N,l}(t):=S_{N,l}(t/s_N)$ of
   $\mu_{N,t/s_N}$ satisfy
$$\tilde S_{N,0}\equiv 1, \quad\quad
	\frac{d}{dt}\tilde S_{N,1}=-\frac{p_N+q_N}{s_N}\tilde S_{N,1}(t)+\frac{a_N(p_N-q_N-b_N(p_N+q_N))}{s_N}\to 0$$
and, for $l\ge2$,	        
\begin{align}\label{moment-ODE-wigner-2}	\frac{d}{dt}	\tilde	S_{N,l}(t)&=
		l\Bigl[	\frac{(p_N-q_N-b_N(p_N+q_N-2(l-1)))a_N}{s_N}\tilde S_{N,l-1}(t)-\frac{p_N+q_N-(l-1)}{s_N}\tilde S_{N,l}(t)\notag\\
                  &-\frac{a_N^2(1-b_N^2)}{s_N}(l-1)\tilde S_{N,l-2}(t)
		  +\frac{Na_N^2}{s_N}(1-b_N^2)\sum_{k=0}^{l-2}\tilde S_{N,k}(t)\tilde S_{N,l-2-k}(t)\notag\\
                  &-\frac{N}{s_N}\sum_{k=0}^{l-2}\tilde S_{N,k+1}(t)\tilde S_{N,l-1-k}(t)
		-2\frac{a_Nb_N N}{s_N}\sum_{k=0}^{l-2}\tilde S_{N,k}(t)\tilde S_{N,l-1-k}(t)\Bigr]\notag\\
                &\overset{N\to\infty}{\sim} l(1-B^2)\sum_{k=0}^{l-2}\tilde S_{N,k}(t)\tilde S_{N,l-2-k}(t).
                	\end{align}
Our starting conditions and induction  show that the 
$\tilde S_{N,l}(t)$ tend  to some  functions $S_l(t)$  with
\begin{equation}\label{recurrence-wigner-limit-2}
  S_0\equiv1,\quad
		S_1(t)=S_1(0)\,,\quad
		S_l(t)=S_l(0)+l(1-B^2)\int_0^t\sum_{k=0}^{l-2}S_k(s)S_{l-2-k}(s)
		\,ds
                \end{equation}
for $l\ge2$ and $t\ge0$. The computations in Section 2 of \cite{VW1} (see in particular the proofs of Lemma 2.4 and Theorem 2.10 there)
now yield the claim similar to the proof of Theorem \ref{semicirc_ODE_thm}.
\end{proof}
Furthermore, with a slight modification in the assumptions:
\begin{theorem}\label{semicirc_ODE_thm3.1}
  Consider sequences $(p_N)_{N\in \mathbb N},(q_N)_{N\in \mathbb N}\subset ]0,\infty[$
    with $\lim_{N\to\infty}(p_N+q_N)/N=\infty$. Let $(b_N)_{N\in \mathbb N}\subset]-1,1[$ be any sequence such that 
    $B:=\lim b_N\in[-1,1]$ exists.
 Let  $(s_N)_{N\in \mathbb N}\subset]0,\infty[$ be a sequence of time scalings  with
    $$\lim_{N\to\infty}\frac{p_N+q_N}{\sqrt{Ns_N}}\geq0,$$
    and define the space scalings
$ a_N:=\sqrt{s_N/N}$. Set
$c:=\lim_{N\to\infty}a_N\left(p_N-q_N-b_N(p_N+q_N)\right)/s_N$.

       Let $\mu\in M^1(\mathbb {R})$ be a starting  measure and  $(x_N)_{N\in\mathbb {N}}$
     starting vectors as in Theorem \ref{semicirc_ODE_thm}.
Let $x_N(t)$ be the solutions of the ODEs \eqref{ODE_main} with $x_N(0)=x_N$ for $N\in\mathbb N$.
Then for all $t>0$, all moments of the empirical measures
\begin{equation*}
	\mu_{N,t/s_N}=\frac{1}{N}\sum_{i=1}^N \delta_{a_N( x_i^N(t/s_N)-b_N)}
\end{equation*}
 tend to those of  $   \mu\boxplus\mu_{sc,2\sqrt{2(1-B^2)t}}\boxplus\delta_{ct}$.
\end{theorem}

In the next step we use the ideas of the proof of Theorem \ref{semicirc_ODE_thm} in combination with Theorem
\ref{main-solutions-ode} which says that the vectors with  the ordered zeros of corresponding Jacobi polynomials
form stationary solutions  of the ODEs (\ref{ODE_main}).  This
leads to the following limit result on the empirical measures
of the zeros of the Jacobi polynomials which was derived in \cite{DS} by different methods:

\begin{theorem}\label{semicirc_ODE_thm-zeros}
Consider sequences $(p_N)_{N\in \mathbb N},(q_N)_{N\in \mathbb N}\subset ]0,\infty[$
    with $\lim_{N\to\infty}p_N/N=\infty$ and  $\lim_{N\to\infty}q_N/N=\infty$ such that
 $C:=\lim_{N\to\infty}p_N/q_N\geq0$ exists. Define
  $$a_N:=\frac{q_N}{\sqrt{Np_N}}, \quad b_N:=\frac{p_N-q_N}{p_N+q_N} \quad(N\in\mathbb N).$$

     Let $-1<z_1^N<\ldots<z_N^N<1$ be the ordered zeros of the Jacobi polynomials $P_N^{(q_N-N,p_N-N)}$.
  Then all moments of the empirical measures
  $$\tilde \mu_{N}	:=\frac{1}{N}\sum_{i=1}^N \delta_{a_N (z_i^N-b_N)}$$
tend to those of $\mu_{sc,4(1+C)^{-3/2}}$. In particular, the $\tilde \mu_{N}$ tend weakly to 
 $\mu_{sc,4(1+C)^{-3/2}}$.
\end{theorem}

\begin{proof} Consider the solutions of the ODEs (\ref{ODE_main})
  as in Theorem \ref{semicirc_ODE_thm} with the initial conditions $x_N:=(b_N,\ldots,b_N)\in A_N$, i.e., with $\mu=\delta_0$ and
  $ S_{N,l}(0)=0$ for $l\ge1$.
  We show that for the moments $\tilde S_{N,l}(t)$ from the proof of Theorem \ref{semicirc_ODE_thm}  the limits
  $\tilde S_{N,l}(\infty):=\lim_{t\to\infty}\tilde S_{N,l}(t)$ exist. In fact, this is clear for $l=0,1$,
and (\ref{moment-ODE-wigner}) and   dominated convergence show
   inductively for $l\ge 2$ that
  \begin{align}\label{moment-ODE-wigner-limit}
    \tilde	S_{N,l}(\infty)&=\frac{l}{p_N+q_N}\lim_{t\to\infty}\int_0^t
 exp\Bigl(-\Bigl(l-\frac{l(l-1)}{p_N+q_N}\Bigr)(t-s)\Bigr) H_{N,l}(s) \> ds\\
&=\frac{l}{p_N+q_N}\lim_{t\to\infty}\int_0^t
 exp\Bigl(-\Bigl(l-\frac{l(l-1)}{p_N+q_N}\Bigr)s\Bigr) H_{N,l}(t-s) \> ds\\
&=\frac{1}{p_N+q_N-l+1}\Biggl(
2a_Nb_N(l-1) \tilde S_{N,l-1}(\infty) \notag\\
&\quad\quad\quad\quad\quad  - (1-b_N^2)a_N^2(l-1) \tilde S_{N,l-2}(\infty)
+Na_N^2(1-b_N^2)\sum_{k=0}^{l-2} \tilde S_{N,k}(\infty) \tilde S_{N,l-2-k}(\infty)
\notag\\
&\quad\quad\quad\quad\quad-
N\sum_{k=0}^{l-2} \tilde S_{N,k+1}(\infty) \tilde S_{N,l-1-k}(\infty)
-2b_NNa_N\sum_{k=0}^{l-2} \tilde S_{N,k}(\infty) \tilde S_{N,l-1-k}(\infty)\Biggr)
 \notag	\end{align}
where $H_{N,l}(s)$ is the term in the big brackets in the last 3 lines of (\ref{moment-ODE-wigner}).
On the other hand, we conclude from Theorem \ref{main-solutions-ode} that
the $\tilde S_{N,l}(\infty)$ are the moments of the empirical measures $\tilde \mu_{N}$.
Furthermore, similar to (\ref{recurrence-wigner-limit}),
 we see that   for all $l$ the limits
$S_l(\infty):=\lim_{N\to\infty} \tilde	S_{N,l}(\infty) $ exist with $S_0(\infty)=1$,  $S_1(\infty)=0$, and
	$$S_l(\infty)= \frac{4l}{(1+C)^3}\sum_{k=0}^{l-2}S_k(\infty)S_{l-2-k}(\infty)  \quad(l\ge2).$$
As this is just the recurrence for the Catalan numbers up to some rescaling
(see e.g.~Section 2.1.1 of \cite{AGZ}), it follows readily that the $S_l(\infty)$ are the moments of
$\mu_{sc,4(1+C)^{-3/2}}$.
\end{proof}

We next turn to the second limit case which concerns Marchenko-Pastur distributions, and which is motivated by 
 Corollary 2.5 of \cite{DS}.  We here assume that the sequences $(p_N)_{N\in \mathbb N},(q_N)_{N\in \mathbb N}$ satisfy
\begin{equation}\label{marchenko_param}
  \lim_{N\to\infty} p_N/N=:\hat{p}\in [1,\infty[,\quad\lim_{N\to\infty}q_N/N=\infty.
\end{equation}
We then choose the norming constants
\begin{equation}\label{marchenko_param1}
  b_N:=-1\,,\quad   a_N:=q_N/N\,.
\end{equation}
In this regime we will obtain a limit theorem which involves Marchenko-Pastur distributions.
 For this we recall that for $c\ge0$, $t>0$, the
 Marchenko-Pastur distribution $\mu_{MP,c,t}\in M^1([0,\infty[)$ is the probability measure
 with $\mu_{MP,c,t}=\tilde{\mu}$ for $c\geq 1$ and $\mu_{MP,c,t}=(1-c)\delta_0+c \tilde{\mu}$ for $0\leq c < 1$, where for
$x_\pm:=t(\sqrt{c}\pm1)^2$, the measure $\tilde{\mu}$ on $]x_-,x_+[$ has the density
 \begin{equation}\frac{1}{2\pi x t}\sqrt{(x_+-x)(x-x_-)}\,.\end{equation}
We also recall (see Exercise 5.3.27 of \cite{AGZ}) that  the R-transforms of the Marchenko-Pastur distributions are given by
\begin{equation}\label{r-transform-mp}
R_{MP,c,t}(z)=\frac{ct}{1-tz}.
\end{equation}
This in particular implies the following well-known relation
\begin{equation}\label{mp-product-formula} \mu_{MP,a,t}\boxplus \mu_{MP,b,t}=\mu_{MP,a+b,t} \quad(a,b,t>0).\end{equation}
We now the following local limit theorem of stationary type which corresponds to Theorem \ref{semicirc_ODE_thm}.

\begin{theorem}\label{MP_ODE_thm}
  Consider $p_N,q_N,a_N,b_N$ as in (\ref{marchenko_param}) and (\ref{marchenko_param1}).
  Let $\mu\in M^1([0,\infty[)$ be a probability measure such that its  moments $c_l$  satisfy
	$\lvert c_l\rvert\leq(\gamma l)^l$ for $l\in\mathbb {N}_0$ with some constant $\gamma>0$.
  Moreover, let $(x_N)_{N\in\mathbb {N}}=((x_{1}^N,\dots,x_{N}^N))_{N\in\mathbb {N}}$   be an associated sequence of starting vectors
  $x_N\in A_N$ as described in Theorem \ref{semicirc_ODE_thm}

Let $x_N(t)$ be the solutions of the ODEs \eqref{ODE_main} with start in $x_N(0)=x_N$ for $N\in\mathbb N, t\ge0$.
	Then for all $t>0$, all moments of the empirical measures
	\begin{equation*}
		\mu_{N,t/(p_N+q_N)}
		=\frac{1}{N}\sum_{i=1}^N \delta_{a_N( x_i^N(t/(p_N+q_N))-b_N)}\end{equation*}
tend to those of the probability measures
\begin{equation}\label{limit-measure-MP-case}
			\mu(t)
				:=\left(\mu_{SC,2\sqrt{2(1-e^{-t})}}
					\boxplus\left(\sqrt{e^{-t}\mu}\right)_{\text{even}}\right)^2
					\boxplus \mu_{MP,\hat{p}-1,2(1-e^{-t})}\,,\quad t>0\,.
		\end{equation}
\end{theorem}

\begin{proof}
As in the proof of Theorem \ref{semicirc_ODE_thm} we see that the recurrence relations \eqref{moment-ODE1}, \eqref{moment-ODE2} together with the initial conditions for $t=0$ and our choice of $b_N$,
  that the moments $\tilde S_{N,l}(t):=S_{N,l}(t/(p_N+q_N))$ of
  $\mu_{N,t/(p_N+q_N)}$ satisfy
\begin{equation*}
	\tilde S_{N,0}\equiv 1, \quad
	\tilde S_{N,1}(t)
	=e^{-t}\left(S_{N,1}(0)-\frac{2a_Np_N}{(p_N+q_N)}\right)+\frac{2a_Np_N}{(p_N+q_N)}
\end{equation*}
and, for $l\ge 2$,
\begin{align}\label{moment-ODE-mp}
	\tilde	S_{N,l}(t)
	&=\exp\Bigl(\Bigl(-l+\frac{l(l-1)}{p_N+q_N}\Bigr)t\Bigr)\Biggl[S_{N,l}(0)\\
		&\quad+\frac{l}{p_N+q_N}\int_0^t \exp\Bigl(\Bigl(l-\frac{l(l-1)}{p_N+q_N}\Bigr)s\Bigr)\Biggl(
		2a_N(p_N-2(l-1)) \tilde S_{N,l-1}(s)\notag\\
		&\quad\quad\quad\quad\quad-N\sum_{k=0}^{l-2} \tilde S_{N,k+1}(s) \tilde S_{N,l-1-k}(s)
		+2a_NN\sum_{k=0}^{l-2} \tilde S_{N,k}(s) \tilde S_{N,l-1-k}(s)
\Biggr)ds\Biggr].
	        \notag
\end{align}
As the starting moments $S_{N,l}(0)$ ($l\ge0$) converge to the corresponding moments of $\mu$
for $N\to\infty$, we conclude by induction on $l$, that the 
$\tilde S_{N,l}(t)$ converge  to some  functions $S_l(t)$ for  $l\ge0$ and $t\ge0$. Moreover, these limits satisfy
\begin{equation}\label{recurrence-mp-limit}
\begin{gathered}
  S_0\equiv1,\quad
		S_1(t)=e^{-t}\left(S_1(0)-2\hat{p}\right)+2\hat{p}\,,\\
		S_l(t)
		=e^{-lt}\left(S_l(0)+2l\int_0^te^{ls}\left(\hat{p}S_{l-1}(s)+
			\sum_{k=0}^{l-2}S_k(s)S_{l-1-k}(s)\right)
		\,ds\right)\,,\;l\geq2\,.
\end{gathered}
\end{equation}
Analogously to the the proof of Theorem \ref{semicirc_ODE_thm} one can show that the $S_l(t)$ satisfy the Carleman condition \eqref{Carleman} for $t>0$. Thus, by the moment convergence theorem there exist unique $\mu_t\in M^1(\mathbb {R})$ with
	$(S_l(t))_l$ as sequences of moments.\\
    To identify the $\mu_t$ we again derive a PDE for the Cauchy and R-transforms of the $\mu_t$.
	We set
	\begin{equation*}
		G(t,z):=G_{\mu_t}(z)
		=\lim_{N\to\infty}G_{\mu_{N,t/(p_n+q_N)}}(z)\,.
	\end{equation*}
	The PDEs \eqref{stieltjes-pde} here lead to the PDE
	\begin{align}
	G_t(t,z)=&zG_z(t,z)+G(t,z)-2(G(t,z)^2+2zG(t,z)G_z(t,z)
		-G_z(t,z))-2\hat{p}G_z(t,z)\notag\\
		=&(z-2(\hat{p}-1)-4zG(t,z))G_z(t,z)+G(t,z)-2G(t,z)^2\,.
	\end{align}
	Using \eqref{R-Cauchy}, we obtain 
	\begin{equation*}
	\begin{split}
		-R_t(t,G(t,z))=&\frac{G_t(t,z)}{G_z(t,z)}\\
		=&(R(t,G(t,z))+\frac{1}{G(t,z)})(1-4G(t,z))-2(\hat{p}-1)\\
                &\quad +(G(t,z)-2G(t,z)^2)(R_z(t,G(t,z))-\frac{1}{G(t,z)^2})\\
		=&-4G(t,z)R(t,G(t,z))-2(\hat{p}-1)-2+(G(t,z)-2G(t,z)^2)R_z(t,G(t,z))+R(t,G(t,z))
	\end{split}
	\end{equation*}
	and thus
	\begin{equation*}
		0=R_t(t,z)-(2z^2-z)R_z(t,z)-(4z-1)R(t,z)-2\hat{p}\,.
	\end{equation*}
	If $\phi(z):=R(0,z)$, the method of characteristics (see e.g. \cite{St}) leads to the solution
\begin{equation}\label{MP_R_lim}
		R(t,z)
		=e^{-t}(1-2z(1-e^{-t}))^{-2}\phi(e^{-t}z(1-2z(1-e^{-t}))^{-1})
			+\frac{2(1-e^{-t})}{1-2(1-e^{-t})z}
			+\frac{2(\hat{p}-1)(1-e^{-t})}{1-2(1-e^{-t})z}\,.
\end{equation}
The third summand on the RHS of this equation corresponds to the second $\boxplus$-summand in 
(\ref{limit-measure-MP-case}). We thus only have to investigate the first two summands on the RHS of (\ref{MP_R_lim}).
For this we fix $s>0$ and define the function $\widehat{\phi}(z):=e^{-s}\phi(e^{-s}z)$.
We also define
\begin{equation*}
f(t,z):=(1-tz)^{-2}\widehat{\phi}\left(\frac{z}{1-tz}\right)+\frac{t}{1-tz}\quad(z\in\mathbb C\setminus \mathbb R, t>0)\,.
\end{equation*}
With the abbreviation $\hat{z}:=\frac{z}{1-tz}$ we then obtain
		\begin{equation*}
		\begin{split}
			f_t(t,z)
			={}&2z(1-tz)^{-3}\widehat{\phi}\left(\hat{z}\right)
				+\frac{z^2}{(1-tz)^{4}}\widehat{\phi}'\left(\hat{z}\right)+\frac{1}{(1-tz)^2}\\
			={}&\frac{2z(1-tz)+2tz^2}{(1-tz)^{3}}\widehat{\phi}\left(\hat{z}\right)
				+\frac{z^2}{(1-tz)^{4}}\widehat{\phi}'\left(\hat{z}\right)+\frac{2zt(1-tz)+t^2z^2+(1-tz)^2}{(1-tz)^2}\\
			={}&\frac{2tz^2}{(1-tz)^{3}}\widehat{\phi}\left(\hat{z}\right)
				+\frac{z^2}{(1-tz)^{4}}\widehat{\phi}'\left(\hat{z}\right)+\frac{t^2z^2}{(1-tz)^2}
				+\frac{2z}{(1-tz)^{2}}\widehat{\phi}\left(\hat{z}\right)+\frac{2zt}{1-tz}+1\\
			={}&z^2f_z(t,z)+2zf(t,z)+1\,.
		\end{split}
		\end{equation*}
Therefore, our $f$ solves the PDE
		\begin{equation}
			f_{t}(t,z)=1+2zf(t,z)+z^2f_z(t,z)\,, \quad
			f(0,z)=R_{\exp(-s)\mu}(z).
		\end{equation}
	 Theorem 4.8 in \cite{VW1} und \eqref{R-scaling} now imply that
		\begin{equation*}
			f(t,z)
			=R_{\left(\mu_{SC,2\sqrt{t}}
				\boxplus \left(\sqrt{\exp(-s)\mu}\right)_{\text{even}}\right)^2}(z)
				\quad\text{for}\,\,t>0\,.
		\end{equation*}
                This and the formula
                $R_{\mu\boxplus\nu}=R_{\mu}+R_{\nu}$ for the R-transform  now complete the proof.
	\end{proof}
	
	\begin{remark}
          If we take the starting distribution $\mu=\mu_{MP,r,s}$ for $r\ge0, s>0$, then,
         with the notations of the preceding proof, $\widehat{\phi}(z)=R_{\mu_{MP,r,s}}(z)=\frac{rs}{1-sz}$.
A partial fraction decomposition here leads to
		\begin{equation*}
			f(t,z)=(1-tz)^{-2}\widehat{\phi}\left(\frac{z}{1-tz}\right)+\frac{t}{1-tz}
			=\frac{rs}{(1-tz)(1-(t+s)z)}+\frac{t}{1-tz}=\frac{r(t+s)}{1-(t+s)z}+\frac{(1-r)t}{1-tz}\,.
		\end{equation*}
	This leads to
		\begin{equation*}
			R_{\left(\mu_{SC,\sqrt{t}}\boxplus\left(\sqrt{\mu_{MP,r,s}}\right)_{\text{even}}\right)^2}(z)
			=\frac{r(t+s)}{1-(t+s)z}+\frac{(1-r)t}{1-tz}\,,\quad r,\,s,\,t\geq0\,
		\end{equation*}
                which generalizes (4.14) in \cite{VW1} slightly.    
	\end{remark}
	
	Similarly to Theorem \ref{semicirc_ODE_thm3} we now consider a variant of  Theorem \ref{MP_ODE_thm} with a different scaling  in space and 
time  where the limit loses its stationary behaviour, and where
the limit corresponds to the results for the Bessel processes of type B and their frozen versions
in Sections 4 and 5 of \cite{VW1}.

\begin{theorem}\label{MP_ODE_thm2}
  Consider sequences $(p_N)_{N\in\mathbb{N}},(q_N)_{N\in\mathbb{N}}\subset]0,\infty]$ with 
  $p_N,q_N>N-1$ 
  for $N\geq1$ and $\lim_{N\to\infty}p_N/N=\hat{p}$. Let  $(s_N)_{N\in \mathbb N}\subset]0,\infty[$ be a sequence of time scalings  with
    $\lim_{N\to\infty}(p_n+q_N)/s_N=0$.  Define the space scalings 
    $a_N:=s_N/N$, $b_N:=-1$ ($N\in\mathbb{N}$).
       Let $\mu\in M^1([0,\infty[)$ be a starting  measure and  $(x_N)_{N\in\mathbb {N}}$
     starting vectors as in Theorem \ref{MP_ODE_thm}.
Let $x_N(t)$ be the solutions of the ODEs \eqref{ODE_main} with $x_N(0)=x_N$ for $N\in\mathbb N$.
Then for all $t>0$, all moments of the empirical measures
\begin{equation*}
		\mu_{N,t/s_N}
		=\frac{1}{N}\sum_{i=1}^N \delta_{a_N( x_i^N(t/s_N)-b_N)}
\end{equation*}
 tend to those of  $\left(\mu_{sc,2\sqrt{2t}}
					\boxplus\left(\sqrt{\mu}\right)_{\text{even}}\right)^2
					\boxplus \mu_{MP,\hat{p}-1,2t}$.
\end{theorem}

\begin{proof}
  The proof is analogous to that of Theorem \ref{MP_ODE_thm}. In fact, 
  the recurrence relations \eqref{moment-ODE1}, \eqref{moment-ODE2} show
   that here the moments $\tilde S_{N,l}(t):=S_{N,l}(t/s_N)$ of
   $\mu_{N,t/s_N}$ satisfy
$$\tilde S_{N,0}\equiv 1, \quad\quad
	\frac{d}{dt}\tilde S_{N,1}=-\frac{p_N+q_N}{s_N}\tilde S_{N,1}(t)+\frac{a_N(p_N-q_N-b_N(p_N+q_N))}{s_N}\to 2\hat{p}$$
and, for $l\ge2$,	        
\begin{align}\label{moment-ODE-MP-2}	\frac{d}{dt}	\tilde	S_{N,l}(t)&=
		l\Bigl[	\frac{(p_N-q_N-b_N(p_N+q_N-2(l-1)))a_N}{s_N}\tilde S_{N,l-1}(t)-\frac{p_N+q_N-(l-1)}{s_N}\tilde S_{N,l}(t)\notag\\
                  &-\frac{a_N^2(1-b_N^2)}{s_N}(l-1)\tilde S_{N,l-2}(t)
		  +\frac{Na_N^2}{s_N}(1-b_N^2)\sum_{k=0}^{l-2}\tilde S_{N,k}(t)\tilde S_{N,l-2-k}(t)\notag\\
                  &-\frac{N}{s_N}\sum_{k=0}^{l-2}\tilde S_{N,k+1}(t)\tilde S_{N,l-1-k}(t)
		-2\frac{a_Nb_N N}{s_N}\sum_{k=0}^{l-2}\tilde S_{N,k}(t)\tilde S_{N,l-1-k}(t)\Bigr]\notag\\
                &\overset{N\to\infty}{\sim} 2l\hat{p}\tilde{S}_{N,l-1}(t)
                +2l\sum_{k=0}^{l-2}\tilde{S}_{N,k}(t)\tilde{S}_{N,l-1-k}(t).
                	\end{align}
Our starting conditions and induction  show that the 
$\tilde S_{N,l}(t)$ tend  to some  functions $S_l(t)$  with
\begin{equation}\label{recurrence-MP-limit-2}
  S_0\equiv1,\quad
		S_1(t)=S_1(0)+2\hat{p}t\,,\quad
		S_l(t)=S_l(0)+2l\int_0^t\left(\hat{p}S_{l-1}(s)+\sum_{k=0}^{l-2}S_k(s)S_{l-1-k}(s)
		\right)ds
                \end{equation}
for $l\ge2$ and $t\ge0$. The computations in Section 4 of \cite{VW1} (see in particular the proofs of Lemma 4.3 and Theorem 4.8 there)
now yield the claim similar to the proof of Theorem \ref{semicirc_ODE_thm}.
\end{proof}

        A slight modification of the proof of Theorem \ref{MP_ODE_thm}  in combination with the assertion about the stationary case in
        Theorem \ref{main-solutions-ode}  leads to the following limit result on the zeros of the Jacobi polynomials
        which was derived in \cite{DS} by different methods. As the modification is completely analogous to the relations between Theorems
        \ref{semicirc_ODE_thm-zeros} and \ref{semicirc_ODE_thm}, we skip the proof.

        \begin{theorem}\label{MP_ODE_thm-zeros}
Consider sequences $p_N,q_N$ with 
$$
  \lim_{N\to\infty} p_N/N=:\hat{p}\in [1,\infty[,\quad\lim_{N\to\infty}q_N/N=\infty,$$
and define the norming constants
 $ b_N:=-1\,,\quad   a_N:=q_N/N$.
 
  Let $-1<z_1^N<\ldots<z_N^N<1$ be the ordered zeros of the Jacobi polynomials $P_N^{(q_N-N,p_N-N)}$.
  Then all moments of the empirical measures
  $$\tilde \mu_{N}	:=\frac{1}{N}\sum_{i=1}^N \delta_{a_N (z_i^N-b_N)}$$
  tend to those of
\begin{equation}\label{limit-measure-MP-case-zero}
		\left(\mu_{SC,2\sqrt{2}}\right)^2
		\boxplus \mu_{MP,\hat{p}-1,2}= \mu_{MP,\hat{p},2}.
		\end{equation}
 In particular, the $\tilde \mu_{N}$ tend weakly to 
   $\mu_{MP,\hat{p},2}$.
\end{theorem}
	
        \section{Almost sure limit theorems for Jacobi processes}
        
In this section we  study the empirical measures of the renormalized Jacobi 
processes $(\tilde{X}_t)_{t\ge0}$ on $A_N$ from the introduction. Recall that these processes satisfy
\begin{equation}\label{SDE_process_normalized_sec4}
	d\tilde{X}_{t,i}
	=\frac{\sqrt{2}}{\sqrt{\kappa}}\sqrt{(1-\tilde{X}_{t,i}^2)}\,dB_{t,i} 
		+\left((p_N-q_N)-(p_N+q_N)\tilde{X}_{t,i}
		+2\sum_{j\colon j\neq i}\frac{1-\tilde{X}_{t,i}\tilde{X}_{t,j}}{\tilde{X}_{t,i}-\tilde{X}_{t,j}}\right)\,dt
\end{equation}
for 
$i=1,\dots,N$ with fixed $\kappa>0$.
\\ Let $a_N\subset]0,\infty[$ and $b_N\subset\mathbb{R}$.
As in  Section 3 we investigate the empirical measures
\begin{equation*}
	\mu_{N,t}
	:=\frac{1}{N}\sum_{i=1}^N\delta_{a_N(\tilde{X}_{t/s_N,i}-b_N)}
\end{equation*}
for $N\to\infty$ for appropriate scalings $a_N,b_N, s_N$.
We begin with the following almost sure  version of Theorem \ref{semicirc_ODE_thm}:

\begin{theorem}\label{semicirc_SDE_thm}
  Consider sequences $(p_N)_{N\in \mathbb N},(q_N)_{N\in \mathbb N}\subset]0,\infty]$
  with $\lim_{N\to\infty}p_N/N=\infty$ and $\lim_{N\to\infty}q_N/N=\infty$ such that
  $C:=\lim_{N\to\infty}p_N/q_N\geq0$ exists. Define
  $$a_N:=\frac{q_N}{\sqrt{Np_N}}, \quad b_N:=\frac{p_N-q_N}{p_N+q_N} \quad(N\in\mathbb N).$$
  
Let $\mu\in M^1(\mathbb {R})$ be a probability measure such that its  moments $c_l$  satisfy
	$\lvert c_l\rvert\leq(\gamma l)^l$ for $l\in\mathbb {N}_0$ with some constant $\gamma>0$.
  Moreover, let $(x^N)_{N\in\mathbb {N}}=((x_{1}^N,\dots,x_{N}^N))_{N\in\mathbb {N}}$   be a sequence of starting vectors
$x_N\in A_N$  such that all moments of the empirical measures
	\begin{equation*}
		\mu_{N,0}:=\frac{1}{N}\sum_{i=1}^N\delta_{a_N( x_i^N-b_N)}
	\end{equation*}
        tend  to those of $\mu$ for $N\to\infty$.
Let $(\tilde{X}^N_{t})_{t\ge0}$ be the solutions of the SDEs \eqref{SDE_process_normalized_sec4} with start in $\tilde{X}^N(0)=x^N$ for $N\in\mathbb N$, $t\ge0$.
	Then for all $t>0$, all moments of the empirical measures
	\begin{equation*}
		\mu_{N,t/(p_N+q_N)}
		=\frac{1}{N}\sum_{i=1}^N \delta_{a_N(\tilde{X}_{t/(p_N+q_N),i}^N-b_N)}\end{equation*}
tend to those of the probability measures	$(e^{-t}\mu)
			\boxplus\left(\sqrt{1-e^{-2t}}\mu_{sc,4(1+C)^{-3/2}}\right)$ almost surely.
\end{theorem}

Before proving this theorem with  the specific  scaling there, we first proceed as in Section 2 and
 investigate
arbitrary affine shifts of $\tilde{X}_t$ first.
For this, define $Y_t:=a_N(\tilde{X}_{t/((p_N+q_N))}-b_N)$ and
\begin{equation*}
	\mu_{N,t}=\frac{1}{N}\sum_{i=1}^N\delta_{Y_{t,i}}\,,\quad
	S_{N,l}(t)=\frac{1}{N}\sum_{i=1}^NY_{t,i}^l\,
\end{equation*}
which fits to the notation in our theorem.
For abbreviation, we now suppress the dependence of $p,q,a,b$ on $N$.
 Then by It\^{o}'s formula
\begin{equation}\label{trans_jacobi_sde}
\begin{split}
	dY_{t,i}
	={}&\sqrt{\frac{2}{\kappa(p+q)}}
		\sqrt{a^2-(Y_{t,i}+ab)^2}\,dB_{t,i}\\
		&+\left[a\left(\frac{p-q}{p+q}-b\right)-Y_{t,i}
		+\frac{2}{p+q}
		\sum_{j\colon j\neq i}\frac{a^2(1-b^2)-Y_{t,i}Y_{t,j}-ab(Y_{t,i}+Y_{t,j})}
			{Y_{t,i}-Y_{t,j}}\right]\,dt\,.
\end{split}
\end{equation}
Furthermore,
for $l\in\mathbb{N}$ we define
\begin{equation}\label{def-martingale}
	M_{l,t}:=\frac{l}{N}\sqrt{\frac{2}{\kappa(p+q)}}
		\int_0^t\sum_{i=1}^NY_{s,i}^{l-1}\sqrt{a^2-(Y_{s,i}+ab)^2}\,dB_{s,i}\,.
\end{equation}
Note that all  $(M_{l,t})_{t\ge0}$ are continuous martingale (w.r.t.~the usual filtration) since 
$\lvert Y_{t,i}\rvert\le a(1+\lvert b\rvert)$ holds for all $i,t$. 
The first empirical moment now satisfies
\begin{equation*}
\begin{split}
	&S_{N,1}(t)-S_{N,1}(0)\\
	={}&\sqrt{\frac{2}{\kappa(p+q)}}\frac{1}{N}\sum_{i=1}^N
		\int_{0}^t\sqrt{a^2-(Y_{s,i}+ab)^2}\,dB_{s,i}\\
		&+\frac{1}{N}\sum_{i=1}^N\int_0^t\left[a\left(\frac{p-q}{p+q}-b\right)-Y_{s,i}
		+\frac{2}{p+q}
		\sum_{j\colon j\neq i}\frac{a^2(1-b^2)-Y_{s,i}Y_{s,j}-ab(Y_{s,i}+Y_{s,j})}
			{Y_{s,i}-Y_{s,j}}\right]\,ds\\
	={}&\int_0^t-S_{N,1}(s)+a\left(\frac{p-q}{p+q}-b\right)\,ds+M_{1,t}\,.
\end{split}
\end{equation*}
This is a linear stochastic differential equation of the form
\begin{equation}\label{linearSDE}
  f(t)-f(0)=\int_0^t(\lambda f(s)+g(s))\,ds+h(t),\end{equation}
  where, in our case, 
$$\lambda=-1,\quad
  f(t)=S_{N,l}(t),,\quad g(t)=a\left(\frac{p-q}{p+q}-b\right),\quad
h(t)=M_{1,t}.$$
As the solution of (\ref{linearSDE}) is given by
\begin{equation}\label{int_eq_lin}
	f(t)
	= e^{\lambda t}\left( f(0)+\int_0^te^{-\lambda s}\left(g(s)+\lambda h(s)\right)\,ds
		\right)+h(t),
\end{equation}
we have
\begin{equation}\label{stoch_moment_ode_sol_1}
	S_{N,1}(t)
	=e^{-t}\left(S_{N,1}(0)+\int_0^te^{s}\left(a\left(\frac{p-q}{p+q}-b\right)-M_{1,s}
		\right)\,ds\right)+M_{1,t}\,.
\end{equation}
By another application of It\^{o}'s formula the higher empirical moments satisfy
\begin{align}\label{stoch_moment_ode}
	&S_{N,l}(t)-S_{N,l}(0)\\
	={}&\frac{1}{N}\sqrt{\frac{2}{\kappa(p+q)}}
		\sum_{i=1}^N\int_{0}^tlY_{s,i}^{l-1}\sqrt{a^2-(Y_{s,i}+ab)^2}\,dB_{s,i}\notag\\
		&+\frac{l}{N}\sum_{i=1}^N\int_0^tY_{s,i}^{l-1}\left[a\left(\frac{p-q}{p+q}-b\right)-Y_{s,i}
		+\frac{2}{p+q}\sum_{j\colon j\neq i}\frac{a^2(1-b^2)-Y_{s,i}Y_{s,j}-ab(Y_{s,i}+Y_{s,j})}
			{Y_{s,i}-Y_{s,j}}\right]ds\notag\\
		&+\frac{1}{N}\sum_{i=1}^N\frac{2}{\kappa(p+q)}
		\int_{0}^tl(l-1)Y_{s,i}^{l-2}\left(a^2-(Y_{s,i}+ab)^2\right)\,ds\notag\\
	={}&M_{l,t}+\int_0^tF_l(s)\,ds
		-\frac{2l(l-1)}{\kappa(p+q)}
		\int_{0}^tS_{N,l}(s)+2abS_{N,l-1}(s)-a^2(1-b^2)S_{N,l-2}(s)\,ds\,,\notag
\end{align}
where by the calculations in \eqref{moment-ODE1} and \eqref{moment-ODE2}
\begin{align*}
	F_l
	={}&- l\Bigl[\left(1-\frac{l-1}{p+q}\right)S_{N,l}
		-a\left(\frac{p-q}{p+q}-b\left(1-2\frac{l-1}{p+q}\right)\right)S_{N,l-1}\\
		&\hphantom{-l}\quad+\frac{a^2(1-b^2)(l-1)}{p+q}S_{N,l-2}
		 -\frac{Na^2(1-b^2)}{p+q}\sum_{k=0}^{l-2}S_{N,k}S_{N,l-2-k}\\
		&\hphantom{-l}\quad+\frac{N}{p+q}\sum_{k=0}^{l-2}S_{N,k+1}S_{N,l-1-k}
			+\frac{2bNa}{p+q}\sum_{k=0}^{l-2}S_{N,k}S_{N,l-1-k}
                  \Bigr].
\end{align*}
Rearranging \eqref{stoch_moment_ode} we obtain
\begin{equation*}
	S_{N,l}(t)-S_{N,l}(0)
	=\int_0^tC_l
		S_{N,l}(s)+f_l(S_{N,1}(s),\dots,S_{N,l-1}(s))\,ds+M_{l,t}\,,
\end{equation*}
with
\begin{equation}\label{const-cl}
  C_l:=-l\left(1+\frac{l-1}{p+q}\left(\frac{2}{\kappa}-1\right)\right)\end{equation}
  and 
\begin{equation*}
\begin{split}
	&f_l(S_{N,1},\dots,S_{N,l-1})\\
	={}&-l\left(-a\left(\frac{p-q}{p+q}
		-b\left(1+\frac{2(l-1)}{p+q}\left(\frac{2}{\kappa}-1\right)\right)\right)S_{N,l-1}
		-\frac{a^2(1-b^2)(l-1)}{p+q}\left(\frac{2}{\kappa}-1\right)S_{N,l-2}\right.\\
	&\hphantom{-l(}\left.-\frac{Na^2(1-b^2)}{p+q}\sum_{k=0}^{l-2}S_{N,k}S_{N,l-2-k}
		+\frac{N}{p+q}\sum_{k=0}^{l-2}S_{N,k+1}S_{N,l-1-k}
		+\frac{2bNa}{p+q}\sum_{k=0}^{l-2}S_{N,k}S_{N,l-1-k}\right)\,.
\end{split}
\end{equation*}
Hence, by \eqref{int_eq_lin}, 
\begin{equation}\label{stoch_moment_ode_sol}
	S_{N,l}(t)
	=e^{C_lt}\left(S_{N,l}(0)+\int_0^te^{-C_ls}
		\left(f_l(S_{N,1}(s),\dots,S_{N,l-1}(s))+C_lM_{l,s}\right)\,ds
		\right)+M_{l,t}\,.
\end{equation}
For the proof of Theorem \ref{semicirc_SDE_thm} and further limit theorems the following observation is crucial.

\begin{lemma}\label{SDE_lem}
	Let $T>0$. Let $p_N,q_N,a_N,b_N$  as in Theorem \ref{semicirc_SDE_thm} or 
	Theorem \ref{MP_SDE_thm} below. Assume that $\lim_{N\to\infty}S_{N,l}(0)$ exists
	for all $l\in\mathbb{N}$.
	Then for all $l\in\mathbb{N}$ the martingales $(M_{l,t})_{t\ge0}$ from (\ref{def-martingale})  converge 
	uniformly to $0$ on $[0,T]$ a.s..
\end{lemma}

\begin{proof}
	In a first step we show that the sequence $(E[\lvert S_{N,l}(t)\rvert])_{N\in\mathbb{N}}$
	is uniformly bounded on $[0,T]$. Here we first study the case 
	$l\in2\mathbb{N}$. By \eqref{stoch_moment_ode_sol} and our assumptions on $p,q,a,b$ it 
	holds, that there are
	non-negative bounded sequences $d_1(N),\dots,d_5(N)$ of numbers such that
	\begin{equation*}
	\begin{split}
		&E(S_{N,l}(t))\\
		\leq{}& e^{C_lt}\left(S_{N,l}(0)
			+\int_0^te^{-C_ls}\left(d_1E\left[\lvert S_{N,l-1}(s)\rvert\right]
			+d_2\left[\lvert S_{N,l-2}(s)\rvert\right]
			+d_3\sum_{k=0}^{l-2}E\left[\lvert S_{N,k}(s)S_{N,l-2-k}(s)\rvert\right]\right.
			\right.\\
			&\left.\left.\hphantom{ e^{C_lt}\left(S_{N,l}(0)+\int_0^t\right.}
			+d_4\sum_{k=0}^{l-2}E\left[\lvert S_{N,k+1}S_{N,l-1-k}(s)\rvert\right]
			+d_5\sum_{k=0}^{l-2}E\left[\lvert S_{N,k}(s)S_{N,l-1-k}(s)\rvert\right]\right)
			\,ds\right)\,.
	\end{split}
	\end{equation*}
	Moreover, by the triangle  inequality and Jensen's inequality, 
	\begin{equation}\label{even-odd}
		\lvert S_{N,l-1}(s)\rvert
		\leq \frac{1}{N}\sum_{i=1}^N\lvert Y_{s,i}\rvert^{l-1}
		\leq \left(\frac{1}{N}\sum_{i=1}^N Y_{s,i}^l\right)^{\frac{l-1}{l}}
		\leq 1+S_{N,l}(s)\,.
	\end{equation}
	By the same reasons, we also have
	\begin{equation*}
		\lvert S_{N,k}(s)S_{N,l-1-k}(s)\rvert
		\leq \left(\frac{1}{N}\sum_{i=1}^N\lvert Y_{s,i}\rvert^{l-1}\right)^{\frac{k}{l-1}}
			\left(\frac{1}{N}\sum_{i=1}^N\lvert Y_{s,i}\rvert^{l-1}
			\right)^{\frac{l-1-k}{l-1}}
		\leq 1+ S_{N,l}(s)\,,
	\end{equation*}
	$\lvert S_{N,k}(s)S_{N,l-2-k}(s)\rvert\leq S_{N,l-2}(s)\leq 1+S_{N,l}(s)$ and
	$\lvert S_{N,k+1}(s)S_{N,l-1-k}(s)\rvert\leq S_{N,l}(s)$. Thus there exist non-negative
	bounded sequences $\tilde{d}_1(N),\tilde{d}_2(N)$ of numbers such that
	\begin{equation*}
		e^{-C_lt}E[S_{N,l}(t)]
		\leq S_{N,l}(0)
			+\int_0^te^{-C_ls}\left(\tilde{d}_1+\tilde{d}_2E[S_{N,l}(s)]\right)\,ds\,.
	\end{equation*}
	By Gronwall's inequality we conclude that
	\begin{equation*}
		e^{-C_{l}t}E[S_{N,l}(t)]
		\leq \left(S_{N,l}(0)
			+\int_0^t\tilde{d}_1e^{-C_ls}\,ds\right)
			\cdot \exp\left(\tilde{d}_2t\right)
	\end{equation*}
	where the  $C_l$ from (\ref{const-cl}) remain bounded. Thus 
	$(E[S_{N,l}(t)])_{N\in\mathbb{N}}$ remains uniformly bounded for  $t\in[0,T]$ in the case of  even $l$.
        Finally, by (\ref{even-odd}) this also holds for $l$ odd.
	\\
	In a second step we now show the claim of the lemma.
        As the Brownian motions $B_i,B_j$ are independent for $i\neq j$, the 
	quadratic variation of $M_{l,t}$ is given by
	\begin{equation*}
		[M_{l}]_t
		=\frac{2}{N^2\kappa(p+q)}\sum_{i=1}^N
			\int_0^tl^2Y_{s,i}^{2l-2}\left(a^2-(Y_{s,i}+ab)^2\right)\,ds\,.
	\end{equation*}
	By the Tchebychev inequality  and the Burkholder-Davis-Gundy inequality there is a constant $c>0$ independent from
	$N$ such that
	\begin{equation*}
	\begin{split}
		P\left(\sup_{0\leq t\leq T}\lvert M_{l,t}\rvert>\epsilon\right)
		\leq{}&\frac{1}{\epsilon^2}E\left[\sup_{0\leq t\leq T}\lvert M_{l,t}\rvert^2\right]\\
		\leq{}&\frac{c}{\epsilon^2}E\left[[M_{l}]_T\right]\\
		={}&\frac{2cl^2}{N^2\kappa(p+q)}\sum_{i=1}^N
			\int_0^TE\left[Y_{s,i}^{2l-2}\left(a^2-(Y_{s,i}+ab)^2\right)\right]
			\,ds\\
		\leq{}&\frac{2cl^2a^2}{N^2\kappa(p+q)}\sum_{i=1}^N
			\int_0^TE\left[Y_{s,i}^{2l-2}\right]
			\,ds\\
		={}&\frac{2cl^2a^2}{N\kappa(p+q)}
			\int_0^TE\left[S_{N,2l-2}(s)\right]\,ds\,.
	\end{split}
	\end{equation*}
	Note that in the case $b_N\equiv 1$ as in Theorem \ref{MP_SDE_thm} we similarly 
	get the bound
	\begin{equation*}
		P\left(\sup_{0\leq t\leq T}\lvert M_{l,t}\rvert>\epsilon\right)
		\leq \frac{4cl^2a}{N\kappa(p+q)}
			\int_0^TE\left[S_{N,2l-1}(s)\right]\,ds\,.
	\end{equation*}
	If we choose $p,q$ and $a$ as in Theorem \ref{semicirc_SDE_thm} we have
	$\frac{a^2}{N(p+q)}\in\mathcal{O}(N^{-2})$. If we choose $p,q$ and $a$ as in Theorem
	\ref{MP_SDE_thm} we have $\frac{a}{N(p+q)}\in\mathcal{O}(N^{-2})$.
	By the first part of the proof we thus conclude that in either case
	$P\left(\sup_{0\leq t\leq T}\lvert M_{l,t}\rvert>\epsilon\right)\in\mathcal{O}(N^{-2})$ for each $\epsilon>0$.
	The claim now follows by the Borel-Cantelli lemma.
\end{proof}

%\begin{theorem}\label{semicirc_SDE_thm}
%  Consider sequences $p_N,q_N,a_N,b_N$ which satisfy the following conditions:
%  \begin{enumerate}
%  \item[\rm{(1)}] $\lim_{N\to\infty}p_N/N=\infty$ and $\lim_{N\to\infty}q_N/N=\infty$;
%  \item[\rm{(2)}]   $\lim_{N\to\infty}q_N/p_N\in[0,\infty]$ exists.
%\end{enumerate} 
%    Put	$a_N:=\left((p_N+q_N)/N\right)^{1/2}$	and $b_N:=\frac{p_N-q_N}{p_N+q_N}$ for $N\in\mathbb N$.
%Let $\mu\in M^1(\mathbb {R})$ be a probability measure such that its  moments $c_l$  satisfy
%	$\lvert c_l\rvert\leq(\gamma l)^l$ for $l\in\mathbb {N}_0$ with some constant $\gamma>0$.
%  Let $(x^N)_{N\in\mathbb {N}}=((x_{1}^N,\dots,x_{N}^N))_{N\in\mathbb {N}}$   be a sequence of starting vectors
%$x^N\in A_N$  such that for all $l\in\mathbb{N}$
%	\begin{equation*}
%		\frac{1}{N}\sum_{i=1}^N(a_N(x^N_{i}-b_N))^l
%		\xrightarrow{N\to\infty}c_l\,.
%	\end{equation*}
%Let $X^N_{t}$ be the solutions of the SDEs \eqref{SDE-alcove-normalized} with start in $X^N(0)=x^N$ for $N\in\mathbb N, t\ge0$.
%	Then for all $t\geq0$,
%	\begin{equation*}
%		\mu_{N,t/(p_N+q_N)}
%		=\frac{1}{N}\sum_{i=1}^N \delta_{a_N \left(X_{t/(p_N+q_N),i}^N-b_N\right)}
%		\xrightarrow{N\to\infty}(e^{-t}\mu)
%			\boxplus\left(\sqrt{1-e^{-2t}}\mu_{sc,\sqrt{2(1-B^2)}}\right)
%	\end{equation*}
%	weakly a.s.,
%	where $B=\lim_{N\to\infty}b_N\in[-1,1]$.
%\end{theorem}

We now turn to the specific scaling in Theorem \ref{semicirc_SDE_thm}:

\begin{proof}[Proof of Theorem \ref{semicirc_SDE_thm}]
	To keep formulas short we again suppress the dependence of $p,q,a,b$ on $N$.
	We define
        $$\mu_t:=(e^{-t}\mu)\boxplus\left(\sqrt{1-e^{-2t}}\mu_{sc,4(1+C)^{-3/2}}\right)$$ 
	with the moments $c_l(t):=\int_{\mathbb{R}}x^l\,d\mu_t(x)$.
	By the proof of Theorem \ref{semicirc_ODE_thm} we have
	$c_1(t)=e^{-t}c_1(0)$ and
	\begin{equation*}
		c_l(t)
		=e^{-lt}\left(c_l(0)+4l(1+C)^{-3}\int_0^te^{ls}\sum_{k=0}^{l-2}c_k(s)c_{l-2-k}(s)\,
			ds\right)\,,
			\quad l\geq2\,.
	\end{equation*}
	By induction we will show that the limits
	$S_l(t):=\lim_{N\to\infty}\int_{\mathbb{R}}x^l\,d\mu_{N,t/(p+q)}(x)$, $l\in\mathbb{N}$, 
	exist and satisfy the same recursion as the $c_l(t)$.\\
	Let $l=1$. By \eqref{stoch_moment_ode_sol_1}, our choice of $b_N$ and
	Lemma \ref{SDE_lem} we have
	$S_1(t):=\lim_{N\to\infty}S_{N,1}(t)=e^{-t}c_l$ locally uniformly in $t$ a.s.\\
	Let $l\geq2$. Note that $C_l$ in \eqref{stoch_moment_ode_sol} converges to $-l$. We now
	calculate the limit of $f_l(S_{N,1}(t),\dots,S_{N,l-1}(t))$. For this note that
	\begin{gather*}
		\lim_{N\to\infty}\frac{4l(l-1)ab}{\kappa(p+q)}=0\,,\;
		\lim_{N\to\infty}\frac{2l(l-1)a^2}{\kappa(p+q)}=0\,,\;
		\lim_{N\to\infty}a\left(\frac{p-q}{p+q}-b\left(1+\frac{2(l-1)}{p+q}
			\left(\frac{2}{\kappa}-1\right)\right)\right)=0
		\,,\\
		\lim_{N\to\infty}\frac{(1-b^2)a^2(l-1)}{p+q}\left(\frac{2}{\kappa}-1\right)=0\,,\;
		\lim_{N\to\infty}N/(p+q)=0\,,\;
		\lim_{N\to\infty}\frac{2bNa}{p+q}=0\,,\\
		\lim_{N\to\infty}\frac{Na^2(1-b^2)}{p+q}=4(1+C)^{-3}\,.\\
	\end{gather*}
	Hence, by our induction assumption, we have a.s. locally uniformly in $t$ that
	\begin{equation*}
		\lim_{N\to\infty}f_l(S_{N,1}(t),\dots,S_{N,l-1}(t))
		=4l(1+C)^{-3}\sum_{k=0}^{l-2}S_k(t)S_{l-2-k}(t)\,.
	\end{equation*}
	Thus by \eqref{stoch_moment_ode_sol} and Lemma \ref{SDE_lem}, the limit 
	$S_l(t)=\lim_{N\to\infty}S_{N,l}(t)$ exists and satisfies
	\begin{equation*}
		S_{l}(t)
		=e^{-lt}\left(S_l(0)+4l(1+C)^{-3}\int_0^te^{ls}\sum_{k=0}^{l-2}S_{k}(s)S_{l-2-k}(s)\,ds
			\right)\;\text{a.s.}\,,
	\end{equation*}
	so that the $S_l(t)$ satisfy the same recursion as the $c_l(t)$.\\
	This proves the claim  in the same way as 
	in the proof of Theorem	\ref{semicirc_ODE_thm}.
\end{proof}

By using the same technique we also readily get the following stochastic version of Theorem \ref{MP_ODE_thm};
please notice that here also Lemma 	\ref{SDE_lem} is available.

\begin{theorem}\label{MP_SDE_thm}
  Consider $p_N,q_N,a_N,b_N$ as in (\ref{marchenko_param}) and (\ref{marchenko_param1}).
  Let $\mu\in M^1([0,\infty[)$ be a probability measure such that its  moments $c_l$  satisfy
	$\lvert c_l\rvert\leq(\gamma l)^l$ for $l\in\mathbb {N}_0$ with some constant $\gamma>0$.
      Moreover, let $(x^N)_{N\in\mathbb {N}}=((x_{1}^N,\dots,x_{N}^N))_{N\in\mathbb {N}}$   be an associated
sequence of starting vectors
$x^N\in A_N$  as the the preceding results.

Let $\tilde{X}^N_{t}$ be the solutions of the SDEs \eqref{SDE_process_normalized_sec4} with start in $\tilde{X}^N(0)=x^N$ for $N\in\mathbb N$, $t\ge0$.
	Then for all $t>0$, all moments of the empirical measures
	\begin{equation*}
		\mu_{N,t/(p_N+q_N)}
		=\frac{1}{N}\sum_{i=1}^N \delta_{a_N(\tilde{X}_{t/(p_N+q_N),i}^N-b_N)}\end{equation*}
tend almost surely to those of the probability measures	
\begin{equation}\label{limit-measure-MP-case-sec4}
			\left(\mu_{SC,2\sqrt{2(1-e^{-t})}}
					\boxplus\left(\sqrt{e^{-t}\mu}\right)_{\text{even}}\right)^2
					\boxplus \mu_{MP,\hat{p}-1,2(1-e^{-t})}\,,\quad t>0\,.
		\end{equation}
\end{theorem}

\begin{remark}
	We point out that by using the methods of the proof as above we also have stochastic
	versions
	of Theorems \ref{semicirc_ODE_thm3}, \ref{semicirc_ODE_thm3.1} and \ref{MP_ODE_thm2}.
	This means that in these theorems the moment convergence holds a.s. if replacing the 
	solution $x(t)$ of \eqref{ODE_main} by the rescaled Jacobi process $\tilde{X}_t$
	satisfying \eqref{SDE_process_normalized_sec4}.
%	The main difference being that one has to replace the $1/(p_N+q_N)$ factors stemming
%	from the time scaling by $1/s_N$.
\end{remark}

For some parameters $\kappa,p,q$, the solutions $(\tilde X_t)_{t\ge0}$ of the SDEs (\ref{SDE_process_normalized_sec4})
admit interpretations in terms of dynamic versions of MANOVA-ensembles over the fields $\mathbb F=\mathbb R,\mathbb C$ by 
Doumerc \cite{Do} as follows. Let $d=1,2$ be the real dimension of $\mathbb F$. Consider  Brownian motions $(Z_t^n)_{t\ge0}$
on the compact groups $SU(n,\mathbb F)$ with some suitable time scalings. Now take positive integers $N,p$ with $N\le p\le n$,
and denote the $N\times p$-block of a square matrix $A$ of size $n$ by $\pi_{N,p}(A)$. Moreover, let $\sigma(B)$ be the ordered spectrum
of some positive semidefinite matrix $B$. It is shown in \cite{Do} that then
$$\Bigl(\tilde X_t:= 2\cdot \sigma\Bigl( \pi_{N,p}(Z_t^N) \pi_{N,p}(Z_t^N)^*\Bigr)-1\Bigr)_{t\ge0}$$
is a diffusion on $A_N$ satisfying the SDE (\ref{SDE_process_normalized_sec4}) with the parameters $p\ge N$, $q:=n-p$, and $\kappa=d/2$.
Clearly, all of the preceding limit results in Section 4 can be applied in this case for suitable sequences  $p_N, n_N  $ of dimension
parameters depending on $N$.

We point out that this geometric interpretation of some Jacobi processes includes the interpretation
for the special case $n=p+N$, i.e., $q=N$, where the
Jacobi processes are  suitable projections of Brownian motions on the compact Grassmann manifolds with the dimension parameters $N,p$ over $\mathbb F$.
We also remark that this even works for the field of quaternions with $\kappa=d/2=2$; see \cite{HS} for the analytical background.

\section{Limit theorems in the noncompact case}

The Jacobi processes on compact alcoves in the preceding section admit analogues in a noncompact setting,
namely the so-called Heckman-Opdam Markov processes associated with root systems of type BC
introduced in Schapira \cite{Sch1, Sch2}.
Due to the close connections with the  Jacobi processes on compact alcoves above, we shall call these processes also
 Jacobi processes in  a noncompact setting.
For some parameters, these processes are related to Brownian motions on noncompact Grassmann manifolds over $\mathbb R, \mathbb C$,
and the quaternions similar to the comments in the end of the preceding section.
For the general background we refer to the monographs \cite{HO, HS} and references therein.

We here  derive analogues of the main results of the Sections 2--4 in this noncompact setting.
For this we first introduce these processes in a way which fits to the compact case. 
We fix some dimension $N\ge2$ and parameters $k_1,k_2\in \mathbb R$ and $k_3>0$ with $k_2\ge0$ and $k_1+k_2\ge 0$.
We define the (noncompact)
Heckman-Opdam Laplacians of type BC on the Weyl chambers
$$\tilde C_N:=\{w\in\mathbb R^N: \> 0\le w_1\le \ldots\le w_N\}$$ of type B by
\begin{align}\label{generator-noncompact-trig}
  L_{trig,k}f(w):=\Delta f(w)  
  +  \sum_{i=1}^N\Biggl(& k_1 \cth(w_i/2)+ 2k_2  \cth(w_i)\\
  &+k_3\sum_{j: j\ne i} \Bigl(\cth(\frac{w_i-w_j}{2})+\cth(\frac{w_i+w_j}{2})\Bigr)\Biggr)f_{x_i}(w)
\notag\end{align}
for functions $f\in C^2(\mathbb R^N)$ which are invariant under the associated Weyl group.
By \cite{Sch1, Sch2}, the  $L_{trig,k}$
are the generators of Feller diffusions $(W_t)_{t\ge0}$ on $C_N$ where the paths are reflected on the boundary.
We next use the transformation $x_i:=\cosh w_i$ ($i=1,\ldots,n$) with
$$x\in C_N:=\{x\in\mathbb R^N: \> 1\le x_1\le \ldots\le x_N\}.$$
The  diffusions $(W_t)_{t\ge0}$ on $\tilde C_N$ then are transformed into  Feller diffusions $(X_t)_{t\ge0}$ on $ C_N$ with 
reflecting boundaries and, by some elementary calculus, with the generators
\begin{equation}\label{generator-noncompact}
  L_kf(x):=\sum_{i=1}^N (x_i^2-1)f_{x_ix_i}(x)+
  \sum_{i=1}^N\Biggl( (k_1+2k_2+ 2k_3(N-1)+1)x_i+k_1
  +2k_3\sum_{j: j\ne i} \frac{x_ix_j-1}{x_i-x_j}\Biggr)f_{x_i}(x).
\end{equation}
As in the introduction, we redefine the parameters by
\begin{equation}\label{parameter-change-k-p-noncompact}
\kappa:=k_3>0, \quad q:= N-1+\frac{1+2k_1+2k_2}{2k_3}, \quad p:= N-1+\frac{1+2k_2}{2k_3}
\end{equation}
with $p, q>N-1$ and rewrite (\ref{generator-noncompact}) as
\begin{equation}\label{generator-noncompact2}
  L_kf(x):=\sum_{i=1}^N (x_i^2-1)f_{x_ix_i}(x)+
 \kappa \sum_{i=1}^N\Biggl((q-p) + (q+p)x_i
  +2\sum_{j: j\ne i} \frac{x_ix_j-1}{x_i-x_j}\Biggr)f_{x_i}(x).
\end{equation}
Moreover, we also consider the  transformed processes
$(\tilde X_{t}:=X_{t/\kappa})_{t\ge0}$ with the generators $\frac{1}{\kappa} L_k$ which then are the
unique strong solutions of the SDEs
\begin{equation}\label{SDE-alcove-normalized-noncompact}
  d\tilde X_{t,i} =\frac{\sqrt 2}{\sqrt\kappa } \sqrt{\tilde X_{t,i}^2-1}\> dB_{t,i} 
+\Bigl((q-p) +(q+p)\tilde X_{t,i} +
2\sum_{j: j\ne i}\frac{\tilde X_{t,i}\tilde X_{t,j}-1}{\tilde X_{t,i}-\tilde X_{t,j}}\Bigr)dt
\end{equation}
for $  i=1,\ldots,N$, a Brownian motion $(B_{t,1},\ldots,B_{t,N})_{t\ge0}$ on $\mathbb R^N$, and starting points  $x_0$ in the interior of $C_N$.

 For $\kappa=\infty$ and $p,q>N-1$, these SDEs degenerate to the ODEs 
\begin{equation}\label{ODE_main-noncompact}
		\frac{d}{dt}x_i(t)
		=(q-p)+(q+p)x_i(t)
			+2\sum_{j: j\neq i}\frac{x_i(t)
			x_j(t)-1}{x_i(t)-x_j(t)} \quad (i=1,\dots,N).
	\end{equation}
Please notice that the RHS of (\ref{ODE_main-noncompact}) is equal to the negative of the  RHS of (\ref{ODE_main}) in the compact case where
the solutions exist on  some different ``complementary'' domain. Theorem \ref{main-solutions-ode} here has the following form; it will be proved in the next section.

  \begin{theorem}\label{main-solutions-ode-noncompact}
          Let $N\in\mathbb{N}$ and $p,q> N-1$. Then for each
		 each $x_0\in C_N$ the ODE \eqref{ODE_main-noncompact} has a unique 
		 solution $x(t)$ for all $t\geq0$ in the following sense:
                 If $x_0$ is in the interior of $\tilde C_N$, then  $x(t)$ exists
                 also in the interior of $ C_N$ for all $t\geq0$.
                 Moreover, for $x_0\in\partial A_N$, 
		there is a unique continuous function $x:[0,\infty)\to C_N$ with $x(0)=x_0$ and
		  $x(t)$ in the interior of $\tilde C_N$ for $t>0$, where $x(t)$ satisfies \eqref{ODE_main-noncompact}.
  \end{theorem}

  For the solutions of (\ref{ODE_main-noncompact})
  we have the following local Wigner-type limit theorem which is completely analogous
to Theorem \ref{semicirc_ODE_thm3}.

  \begin{theorem}\label{semicirc_ODE_thm-noncompact}
Consider sequences $(p_N)_{N\in \mathbb N},(q_N)_{N\in \mathbb N}\subset ]0,\infty[$
    with $p_N,q_N>N-1$ for $N\ge 1$. Let $(b_N)_{N\in \mathbb N}\subset]1,\infty[$ be a sequence such that 
    $B:=\lim b_N\in[1,\infty]$ exists.

 Let  $(s_N)_{N\in \mathbb N}\subset]0,\infty[$ be a sequence of time scalings  with
    $$\lim_{N\to\infty}\frac{p_N+q_N}{\sqrt{Ns_N}}=0,$$
    and define the space scalings
$ a_N:=\sqrt{s_N/N}$.

       Let $\mu\in M^1(\mathbb {R})$ be a starting  measure such that its  moments $c_l$  satisfy
       $\lvert c_l\rvert\leq(\gamma l)^l$ for $l\in\mathbb {N}_0$ with some constant $\gamma>0$.
       Let $(x_N)_{N\in\mathbb {N}}$ be associated starting vectors with $x_N\in C_N$  as the preceding limit results.

       Let $x_N(t)$ be
       the solutions of the ODEs \eqref{ODE_main} with $x_N(0)=x_N$ for $N\in\mathbb N$.
Then for all $t>0$, all moments of the empirical measures
\begin{equation*}
		\mu_{N,t/(p_N+q_N)}
		=\frac{1}{N}\sum_{i=1}^N \delta_{a_N( x_i^N(t/s_N)-b_N)}\end{equation*}
 tend to those of  $   \mu\boxplus\mu_{sc,2\sqrt{(B^2-1)t}}$.
\end{theorem}
	
\begin{proof} As the RHSs of (\ref{ODE_main-noncompact}) and  (\ref{ODE_main}) are equal up to a sign, the computations in Section 2 and 
in the proof of Theorem \ref{semicirc_ODE_thm3} imply that for  $l\ge0$ and $t\ge0$, the moments $\tilde S_{N,l}(t)$
of the empirical measures
$\mu_{N,t/s_N}$ converge for $N\to\infty$  to functions $S_l(t)$  which satisfy
\begin{equation}\label{recurrence-wigner-limit-noncompact}
  S_0\equiv1,\quad
		S_1(t)=S_1(0)\,,\quad
		S_l(t)=S_l(0)+l(B^2-1)\int_0^t\sum_{k=0}^{l-2}S_k(s)S_{l-2-k}(s)
		\,ds \quad(l\ge2).
                \end{equation}
The claim now follows in the same way as in Theorem \ref{semicirc_ODE_thm3}.
\end{proof}

The stationary local limit Theorem \ref{semicirc_ODE_thm}
does not seem to have a meaningful analogue in the noncompact setting, as 
 the assumptions on the $p_N, q_N, a_N, b_N$ in Theorem \ref{semicirc_ODE_thm}  imply that $b_N\in]-1,1[$ holds for all $N$
                  such that the rescaled empirical measures
measures for $t=0$ in the assumptions of Theorem \ref{semicirc_ODE_thm} cannot converge.

On the other hand, we have the following variants of the stationary Theorem \ref{MP_ODE_thm} as well as of the non-stationary Theorem
\ref{MP_ODE_thm2}
both of which involve Marchenko-Pastur distributions. Note that due to the time-inversion also the analogue to Theorem \ref{MP_ODE_thm} is now non-stationary:

\begin{theorem}\label{MP_ODE_thm_noncompact}
  Consider sequences $(p_N)_{N\in\mathbb{N}},(q_N)_{\mathbb{N}}\subset]0,\infty]$ with
    $$\lim_{N\to\infty}p_N/N=\infty \quad\text{ and}\quad  \lim_{N\to\infty}q_N/N=\hat{q}.$$
    Define
  $a_N:=p_N/N$, $b_N:=1$ $(N\in\mathbb{N})$.
  Let $\mu\in M^1([0,\infty[)$ be a probability measure such that its  moments $c_l$  satisfy
	$\lvert c_l\rvert\leq(\gamma l)^l$ for $l\in\mathbb {N}_0$ with some constant $\gamma>0$.
  Moreover, let $(x_N)_{N\in\mathbb {N}}$   be an associated sequence of starting vectors
$x_N\in C_N$  as in the  preceding limit results.
Let $x_N(t)$ be the solutions of the ODEs \eqref{ODE_main-noncompact} with start in $x_N(0)=x_N$ for $N\in\mathbb N, t\ge0$.
	Then for all $t>0$, all moments of the empirical measures
	\begin{equation*}
		\mu_{N,t/(p_N+q_N)}
		=\frac{1}{N}\sum_{i=1}^N \delta_{a_N( x_i^N(t/(p_N+q_N))-b_N)}\end{equation*}
tend to those of the probability measures
\begin{equation}\label{limit-measure-MP-case_noncompact}
			\mu(t)
				:=\left(\mu_{SC,2\sqrt{2(e^t-1})}
					\boxplus\left(\sqrt{e^{t}\mu}\right)_{\text{even}}\right)^2
					\boxplus \mu_{MP,\hat{q}-1,2(e^t-1)}\,,\quad t>0\,.
\end{equation}
\end{theorem}

\begin{proof}
	The proof is completely analogous to the one of Theorem \ref{MP_ODE_thm}. We just
	give the main steps.\\
	The moments $\tilde{S}_{N,l}(t)$ of the empirical measures $\mu_{N,t/(p_N+q_N)}$
	converge for $N\to\infty$ to functions $S_l(t)$ which satisfy
	\begin{equation}\label{recurrence-mp-limit_noncompact}
\begin{gathered}
  S_0\equiv1,\quad
		S_1(t)=e^{t}\left(S_1(0)-2\hat{q}\right)+2\hat{q}\,,\\
		S_l(t)
		=e^{lt}\left(S_l(0)+2l\int_0^te^{-ls}\left(\hat{q}S_{l-1}(s)+
			\sum_{k=0}^{l-2}S_k(s)S_{l-1-k}(s)\right)
		\,ds\right)\,,\;l\geq2\,.
\end{gathered}
\end{equation}
    Denote the Cauchy-transform of the limiting measure 
    $\mu_t:=\lim_{N\to\infty}\mu_{N,t/(p_N+q_N)}$ by
	\begin{equation*}
		G(t,z):=G_{\mu_t}(z)
		=\lim_{N\to\infty}G_{\mu_{N,t/(p_n+q_N)}}(z)\,.
	\end{equation*}
As for the PDEs \eqref{stieltjes-pde}, but with an additional minus sign, this leads to the PDE
	\begin{align}
	G_t(t,z)
	%=&-zG_z(t,z)-G(t,z)-2(G(t,z)^2+2zG(t,z)G_z(t,z)	-G_z(t,z))-2\hat{q}G_z(t,z)\notag\\
		=&(-z-2(\hat{q}-1)-4zG(t,z))G_z(t,z)-G(t,z)-2G(t,z)^2\,.
	\end{align}
	Using \eqref{R-Cauchy}, we obtain for the R-transforms that
	\begin{equation*}
		 0=R_t(t,z)-(z+2z^2)R_z(t,z)-2\hat{q}-(4z+1)R(t,z)\,.
	\end{equation*}
	If we put $\phi(z):=R(0,z)$, the method of characteristics here leads to 
\begin{equation}\label{MP_R_lim_noncompact}
		R(t,z)
		=e^{t}(1-2z(e^t-1))^{-2}\phi(e^{t}z(1-2z(e^t-1))^{-1})
			+\frac{2(e^t-1)}{1-2z(e^t-1)}
			+\frac{2(\hat{q}-1)(e^t-1)}{1-2z(e^t-1)}\,.
\end{equation}
Finally if we set $\widehat{\phi}(z):=e^{s}\phi(e^{s}z)$ and 
\begin{equation*}
f(t,z):=(1-tz)^{-2}\widehat{\phi}\left(\frac{z}{1-tz}\right)+\frac{t}{1-tz}\quad(z\in\mathbb C\setminus \mathbb R, t>0)\,,
\end{equation*}
the claim now follows as in the proof of Theorem \ref{MP_ODE_thm}.
\end{proof}

The following result also follows in the same way by the methods of the proof of  Theorem \ref{MP_ODE_thm2}.

\begin{theorem}\label{MP_ODE_thm2_noncompact}
  Consider sequences $(p_N)_{N\in\mathbb{N}},(q_N)_{N\in\mathbb{N}}\subset]0,\infty]$ with $p_N,q_N>N-1$ 
  for $N\geq1$ and with $\lim_{N\to\infty}q_N/N=\hat{q}\in [1,\infty[$. Let  $(s_N)_{N\in \mathbb N}\subset]0,\infty[$ be a sequence of time scalings  with
    $\lim_{N\to\infty}(p_n+q_N)/s_N=0$.  Define the space scalings 
    $a_N:=s_N/N$, $b_N:=1$ ($N\in\mathbb{N}$).
       Let $\mu\in M^1([0,\infty[)$ be a starting  measure and  $(x_N)_{N\in\mathbb {N}}$
  associated   starting vectors as before.
Let $x_N(t)$ be the solutions of the ODEs \eqref{ODE_main-noncompact} with $x_N(0)=x_N$ for $N\in\mathbb N$.
Then for all $t>0$, all moments of the empirical measures
\begin{equation*}
		\mu_{N,t/s_N}
		=\frac{1}{N}\sum_{i=1}^N \delta_{a_N( x_i^N(t/s_N)-b_N)}
\end{equation*}
 tend to those of  $\left(\mu_{sc,2\sqrt{2t}}
					\boxplus\left(\sqrt{\mu}\right)_{\text{even}}\right)^2
					\boxplus \mu_{MP,\hat{q}-1,2t}$.
\end{theorem}

We finally mention that also the stochastic limit results from Section 4,
that correspond to the deterministic limit Theorems \ref{semicirc_ODE_thm-noncompact}-\ref{MP_ODE_thm2_noncompact},
can be transferred to the noncompact setting. We here skip the details.

\section{Appendix: Solutions of the differential equations with start on the singular boundary}
	
In this section we  prove Theorems \ref{main-solutions-ode} and \ref{main-solutions-ode-noncompact}.

We first study  the ODE (\ref{ODE_main})
which has the form
	\begin{equation}\label{ODE_main_app}
	\frac{d}{dt}x_i(t)
		=(p-q)-(p+q)x_i(t)
			+2\sum_{\substack{j:    j\neq i}}^N\frac{1-x_i(t)
			x_j(t)}{x_i(t)-x_j(t)}\,,\quad i=1,\dots,N.
	\end{equation}
        In order to prove parts of  Theorem \ref{main-solutions-ode},
        it is useful to interpret this ODE as a gradient system; see e.g. Section 9.4 of \cite{HiS}
          on the background.
          However, it can be easily checked that (\ref{ODE_main_app}) is not a gradient system.
          In order to obtain a gradient system, we use the transformation 
          $x_i=:\cos \tau_i$ with $\pi\ge \tau_1\ge\ldots\ge \tau_N\ge0$ which is motivated by
          the theory of Heckman-Opdam hypergeometric functions in \cite{HO, HS} in its trigonometric form (see also the introduction),
          and which is also useful 
in \cite{HV} for nice covariance matrices in some freezing central limit theorem. 
In fact, elementary calculus shows that (\ref{ODE_main_app}) is equivalent to the ODE
        \begin{equation}\label{trig-ode}
        \begin{split}
        \frac{d}{dt} \tau_i(t)={}&
           (q-p)\cot\left(\frac{\tau_i(t)}{2}\right)+2(p+1-N) \cot(\tau_i(t))\\
          &+\sum_{j: j\ne i}\left(\cot\left(\frac{\tau_i(t)-\tau_j(t)}{2}\right)+
          \cot\left(\frac{\tau_i(t)+\tau_j(t)}{2}\right)\right)
          \end{split}
          \end{equation}
        for $i=1,\ldots,N$ which is a gradient system.
        In fact, if 
        $V(\tau):= \ln \tilde V(\tau)$ with
    \begin{equation}\label{potential}      \tilde V(\tau):=\left(\prod_{i=1}^N \sin(\tau_i/2)\right)^{2(q-p)}\cdot
\left(\prod_{i=1}^N \sin(\tau_i)\right)^{2(p+1-N)}\cdot\prod_{i,j: \> i<j} \left(\sin\left(\frac{\tau_i-\tau_j}{2}\right)\sin\left(\frac{\tau_i+\tau_j}{2}\right)\right)^2,
 \end{equation}
    then (\ref{trig-ode}) has the form $	\frac{d}{dt} \tau(t)=\text{grad} \> V( \tau(t))$ with $\tau=(\tau_1,\ldots,\tau_N)$.

    We next search for a maximum of the potential $V$, i.e., of $\tilde V$.
    For this we observe  that, with some constant $C$,
    $$\tilde V(\tau)= C\cdot \prod_{i=1}^N((1-x_i)^{q+1-N} (1+x_i)^{p+1-N}) \cdot\prod_{i,j: \> i<j} (x_i-x_j)^2.$$
    A classical result of Stieltjes (see Section 6.7  of \cite{S}) now shows that for $\pi> \tau_1>\ldots> \tau_N>0$,
    this expression has a unique maximum for $x=z$ where  the vector $z$ consists of the ordered roots of the Jacobi polynomial
    $P_N^{(q-N,p-N)}$. Therefore, Section 9.4 of \cite{HiS} yields the following part of Theorem 
    \ref{main-solutions-ode}:

	\begin{lemma}\label{ODE_lem_app}
		Let $N\in\mathbb{N}$ and $p,q> N-1$.
		For each $x_0\in\inte A_N$ the ODE \eqref{ODE_main_app} has a unique 
		solution $x(t)$
		with $x(t)\in\inte A_N$ for all $t\geq0$.
Moreover, $\lim_{t\to\infty}x(t)=z$ where $z\in\inte A_N$ is the 
		vector consisting of the ordered roots of  $P_N^{(q-N,p-N)}$.
                \end{lemma}
	
In order to complete the proof of Theorem 
\ref{main-solutions-ode}, we still have to prove the following theorem. Its proof is an adaptation of the corresponding results for
          the Hermite- and Laguerre case in \cite{VW2}.

\begin{theorem}\label{theorem-appendix}
		Let $N\in\mathbb{N}$ and $p,q>N-1$.
		For each $x_0\in \partial A_N$ the ODE \eqref{ODE_main_app} has a unique 
		solution $x(t)$ for all $t\geq0$ in the following sense: For each $x_0\in\partial A_N$
		there is a continuous function $x:[0,\infty)\to A_N$ with $x(0)=x_0$ such that
		$x(t)\in\inte A_N$ for all $t>0$ and
		$x\colon(0,\infty)\to\inte A_N$ satisfies \eqref{ODE_main_app}.
		Moreover, 
		$\lim_{t\to\infty}x(t)=z$ with $z\in A_N$ as above.
\end{theorem}

	\begin{proof} We  use of the elementary symmetric polynomials $e_n^m$
	($n=0,\dots,m$) in $m$ variables which are characterized by 
	\begin{equation*}
		\prod_{j=1}^m(z-x_j)=\sum_{j=0}^m(-1)^{m-j}e_{m-j}^m(x)z^j\,,
			\quad z\in\mathbb{C},\quad x=(x_1,\dots,x_m)\,.\end{equation*}
	  Consider the map $e:A_N\to\mathbb{R}^N$, $e(x)=(e_1^N(x),\dots,e_N^N(x))$. Then
		$e:A_N\to e(A_N)$ is a homeomorphism, and
		$e:\inte A_N\to e(\inte A_N)$ is a diffeomorphism.
		We will use  the following notation: Let $x\in\mathbb{R}^N$ and 
		$S\subseteq\{1,\dots,N\}$ a nonempty set. Denote by
		$x_S\in\mathbb{R}^{\lvert S\rvert}$ the vector with coordinates $x_i$, $i\in S$,
		in the natural ordering on $S$. With this convention we have
		\begin{equation*}
			\sum_{i=1}^Ne_{k-1}^{N-1}(x_{\{1,\dots,N\}\setminus\{i\}})=(N-k+1)e_{k-1}^N(x)\,,
			\quad
			\sum_{i=1}^ne_{k-1}^{N-1}(x_{\{1,\dots,N\}\setminus\{i\}})x_i
			=ke_k^N(x),
		\end{equation*}
		and 
		\begin{equation*}
			e_{k-1}^{N-1}(x_{\{1,\dots,N\}\setminus\{i\}})
			-e_{k-1}^{N-1}(x_{\{1,\dots,N\}\setminus\{j\}})
			=-(x_i-x_j)e_{k-2}^{N-2}(x_{\{1,\dots,N\}\setminus\{i,j\}})\,.
		\end{equation*}
		Hence
		\begin{equation*}
		\begin{split}
			\sum_{i,j=1: \> i\neq j}
				\frac{e_{k-1}^{N-1}(x_{\{1,\dots,N\}\setminus\{i\}})}{x_i-x_j}
			&=\sum_{{i,j=1: \> i<j}}
				\frac{e_{k-1}^{N-1}(x_{\{1,\dots,N\}\setminus\{i\}})
				-e_{k-1}^{N-1}(x_{\{1,\dots,N\}\setminus\{j\}})}{x_i-x_j}\\
			&=-\sum_{i,j: \> i<j}e_{k-2}^{N-2}(x_{\{1,\dots,N\}\setminus\{i,j\}})
				=-\frac{(N-k+2)(N-k+1)}{2}e_{k-2}^N(x)
		\end{split}
		\end{equation*}
		and
		\begin{equation*}
		\begin{split}
			\sum_{i,j=1: \> i\neq j}
				\frac{e_{k-1}^{N-1}(x_{\{1,\dots,N\}\setminus\{i\}})x_ix_j}{x_i-x_j}
			&=\sum_{i,j=1: \> i<j}
				\frac{e_{k-1}^{N-1}(x_{\{1,\dots,N\}\setminus\{i\}})
				-e_{k-1}^{N-1}(x_{\{1,\dots,N\}\setminus\{j\}})}{x_i-x_j}x_ix_j\\
			&=-\sum_{i,j=1: \> i<j}e_{k-2}^{N-2}(x_{\{1,\dots,N\}\setminus\{i,j\}})
				x_ix_j
			=-\frac{k(k-1)}{2}e_k^N(x)\,.
		\end{split}
		\end{equation*}
		By transforming \eqref{ODE_main_app} with the homeomorphism $e$ we  get the ODEs
		\begin{align}\label{transformed_ODE}
			\frac{d}{dt}e_1^N(x(t))
			=&\sum_{i=1}^N\frac{d}{dt}x_i(t)
			=N(p-q)-(p+q)e_1^N(x(t))\,,\notag\\
			\frac{d}{dt}e_k^N(x(t))
			=&\sum_{i=1}^Ne_{k-1}^{N-1}(x_{\{1,\dots,N\}\setminus\{i\}}(t))
				\left((p-q)-(p+q)x_i(t)
				+2\sum_{{j=1: \> j\neq i}}
				\frac{1-x_i(t)x_j(t)}{x_i(t)-x_j(t)}\right)\notag\\
%			={}&(N-k+1)(p-q)e_{k-1}^N(x(t))-k(p+q)e_k^N(x(t))
%				-(N-k+2)(N-k+1)e_{k-2}^N(x(t))
%				\\
%				&-2\sum_{{i,j: \> i\neq j}}
%				\frac{e_{k-1}^{N-1}(x_{\{1,\dots,N\}\setminus\{i\}}(t))x_i(t)x_j(t)}
%				{x_i(t)-x_j(t)}\\
			={}&k(-(p+q)+k-1)e_k^N(x(t))+(N-k+1)(p-q)e_{k-1}^N(x(t))\notag\\
				&-(N-k+2)(N-k+1)e_{k-2}^N(x(t))\,,\quad k\in\{2,\dots,N\}\,.
		\end{align}
		These are linear differential equations of the type $f'(t)=\lambda f(t)+g(t)$ with the solutions
		$f(t)=e^{\lambda t}\left(f(0)+\int_0^te^{-\lambda s}g(s)\,ds\right)$.
		Thus,
                \begin{align}\label{solutions-symmetric}
		e_1^N(x(t))&=e^{-(p+q)t}\left(e_1^N(x_0)-N\frac{p-q}{p+q}\right)+N\frac{p-q}{p+q} \quad\text{and}\notag\\
			e_k^N(x(t))&=
			e^{c_kt}\left(\vphantom{\int_0^1}e_k^N(x_0)\right.\\
				&\left.+\int_0^te^{-c_ks}\left((N-k+1)(p-q)e_{k-1}^N(x(s))
				-(N-k+2)(N-k+1)e_{k-2}^N(x(s))\right)\,ds\right)\,,
		\notag\end{align}
		where $c_k=k(-(p+q)+k-1)<0$, $k\in\{2,\dots,N\}$.
		By induction we see  that each $e_k^N(x(t))$ is a linear combination of terms of
		the form $e^{rt}$, $r\leq0$. Thus the limits
		$\hat{e}_k:=\lim_{t\to\infty}e_k^N(t)$ exist. We  claim that  $\hat{e}=e(z)$ holds.
                To prove this we observe from the limit assertion in Lemma \ref{ODE_lem_app} that this holds for all
starting points  $x_0\in\inte A_N$. Furthermore, as  $\hat{e}$ depends continuously on  $x_0$ by (\ref{solutions-symmetric}), we obtain
$\hat{e}=e(z)$ also for   $x_0\in\partial A_N$.

		We now turn to the case $x_0\in\partial A_N$. Clearly, as $e$
		is injective there exists at most one solution of \eqref{ODE_main_app}.
		For the existence of a solution we claim that the inverse mapping of $e$
		transforms solutions of \eqref{transformed_ODE} back into solutions of
		\eqref{ODE_main_app} in the sense of the theorem. For this we prove that for any
		starting point $x_0\in\partial A_N$ in \eqref{ODE_main} and its image
		$e(x_0)$ the solution $\tilde{e}(t)$, $t\geq0$, of the ODEs 
		\eqref{transformed_ODE} with $\tilde{e}(0)=e(x_0)$ satisfies
		$\tilde{e}(t)\in e\left(\inte A_N\right)$ for all $t>0$. If this is
		shown it follows that the preimage of $(\tilde{e}(t))_{t\geq0}$ under $e$ solves
		\eqref{ODE_main_app}.
                
		To prove this, we recapitulate that for each starting point in 
		$e\left(\inte A_n\right)$ the solution $\tilde{e}$ of  \eqref{transformed_ODE} satisfies
		$\tilde{e}(t)\in e\left(\inte A_n\right)$ for all $t\geq0$, and
		that for all fixed $t\geq0$ the solutions $\tilde{e}(t)$ depend continuously 
		on arbitrary starting points in $\mathbb{R}^N$ by a classical result on ODEs. 
		Hence, for each starting point $\tilde{e}(0)\in e(A_N)$ we have
		$\tilde{e}(t)\in e(A_N)$ for $t\geq0$.
                
		Assume that there is a starting point $x_0\in\partial A_N$ and some $t_0>0$
		such that the solution $(\tilde{e}(t))_{t\geq0}$ of \eqref{transformed_ODE}
		with start at $e(x_0)$ satisfies
		\begin{equation}\label{thm_app_eq1}
			\tilde{e}(t)\notin e\left(\inte A_n\right)\,,\quad t\in[0,t_0]\,.
		\end{equation}
		For $x=(x_1,\dots,x_n)\in\mathbb{R}^N$ we define the
		discriminant
		\begin{equation}\label{discriminant}
			D(x):=\prod_{i=1}^N(1-x_i^2)\cdot
				\prod_{\substack{i,j=1\\i\neq j}}^N(x_j-x_i)\,.
		\end{equation}
		$D$ is a symmetric polynomial in $x_1,\dots,x_N$ and thus, by a classical result
		on elementary polynomials, a polynomial $\tilde{D}$ in $e_1^N(x),\dots,e_N^N(x)$.
		By \eqref{thm_app_eq1} we thus  deduce
		\begin{equation*}
			\tilde{e}(t)\in e(\partial A_N)\subseteq 
			Y:=\{y\in\mathbb{R}^N\colon\,\tilde{D}(y)=0\}\,,\quad t\in[0,t_0]\,.
		\end{equation*}
		We obtain that
		$\tilde{D}(\tilde{e}(t))=0$ for $t\in[0,t_0]$. As $\tilde{D}(\tilde{e}(t))$
		is a linear combination of terms of the form $e^{rt}$ with $r\leq0$ it follows
		that $\tilde{D}(\tilde{e}(t))=0$ for all $t\geq0$. As
		$Y\cap e\left(\inte A_n\right)=\emptyset$, we conclude that
		$\tilde{e}(t)\notin e\left(\inte A_n\right)$ for all $t\geq0$.
		But this is a contradiction to
		$\lim_{t\to\infty}\tilde{e}(t)=e(z)\in e\left(\inte A_n\right)$.
		Hence  $\tilde{e}(t)\in e\left(\inte A_n\right)$ for  $t>0$ as claimed. This completes the proof.
	\end{proof}

We finally turn to the proof of Theorem    \ref{main-solutions-ode-noncompact} on the ODEs (\ref{ODE_main-noncompact}).
We proceed as in the proof of    Theorem      \ref{main-solutions-ode} and notice first that the
transform $x_i=\cosh \tau_i$ ($i=1,\ldots,N$) transform the ODEs  (\ref{ODE_main-noncompact}) again into some gradient system.
As for Lemma \ref{ODE_lem_app}, we thus obtain:

\begin{lemma}\label{ODE_lem_app-nonvompact}
		Let $N\in\mathbb{N}$ and $p,q> N-1$.
		For each $x_0\in\inte C_N$ the ODE \eqref{ODE_main-noncompact} has a unique 
		solution $x(t)$
		with $x(t)\in\inte C_N$ for  $t\geq0$.
\end{lemma}

\begin{proof} We only have to check that the system is not explosive in finite time. For this we again use the elementary symmetric polynomials $e_n^m$
	as well as the homeomorphism $e:C_N\to e(C_N)\subset\mathbb{R}^N$ with $e(x)=(e_1^N(x),\dots,e_N^N(x))$ as in the proof 
 of Theorem \ref{theorem-appendix}.
 As the right hand sides of the ODEs (\ref{ODE_main}) and (\ref{ODE_main-noncompact}) are equal up to a sign change,
 we conclude from the computations  in the proof 
 of Theorem \ref{theorem-appendix} (see in particular (\ref{solutions-symmetric})) that
                \begin{align}\label{solutions-symmetric-noncompact}
		e_1^N(x(t))&=e^{(p+q)t}\left(e_1^N(x_0)-N\frac{p-q}{p+q}\right)+N\frac{p-q}{p+q},\notag\\
			e_k^N(x(t))&=
			e^{c_kt}\left(\vphantom{\int_0^1}e_k^N(x_0)\right.\\
				&\left.+\int_0^te^{-c_ks}\left((N-k+1)(p-q)e_{k-1}^N(x(s))
				-(N-k+2)(N-k+1)e_{k-2}^N(x(s))\right)\,ds\right)\,,
		\notag\end{align}
                with $c_k=k((p+q)+1-k)<0$ for $k=2,\dots,N$. In summary, $e(x(t))$
                satisfies some linear ODE and exists thus for all $t\ge0$. The claim now follows by a transfer back to $\inte C_N$.
\end{proof}

To complete the proof of Theorem 
\ref{main-solutions-ode-noncompact}, we prove the following analogue of Theorem \ref{theorem-appendix}.

\begin{theorem}\label{theorem-appendix-noncompact}
		Let $N\in\mathbb{N}$ and $p,q>N-1$.
		For each starting value $x_0\in \partial A_N$ the ODE  (\ref{ODE_main-noncompact}) has a unique 
		solution $x(t)$ for  $t\geq0$ in the sense as described in Theorem \ref{theorem-appendix}.
\end{theorem}

\begin{proof} We use the notatons of the proof of Lemma \ref{ODE_lem_app-nonvompact} and
 consider some starting point $x_0\in\partial C_N$. 
		For the existence of a solution we claim that the inverse mapping of $e$
		transforms the functions in (\ref{solutions-symmetric-noncompact}) back into solutions of
		(\ref{ODE_main-noncompact})
                in the sense of the theorem, i.e.,
                that $\tilde{e}(t):=(e_1^N(x(t)),\ldots,e_N^N(x(t))\in e\left(\inte A_N\right)$ holds for all $t>0$.
                To prove this, we have to check that	$\tilde{e}(t)\notin e(\partial A_N)$ for $t>0$.

		Assume that for some $x_0\in\partial A_N$ and  $t_0>0$
		we have 
		$\tilde{e}(t)\notin e\left(\inte A_n\right)$ for $ t\in[0,t_0]$.
                We now use the	discriminant $D$ from (\ref{discriminant}) as well as $\tilde D$ there. We
                conclude from the corresponding  methods in
 in the proof of Theorem \ref{theorem-appendix} that $\tilde{D}(\tilde{e}(t))=0$ for $t\in[0,t_0]$ implies that
$\tilde{D}(\tilde{e}(t))=0$ for all $t\in\mathbb R$.
We now recapitulate that the solutions (\ref{solutions-symmetric-noncompact}) and (\ref{solutions-symmetric})
of the corresponding ODEs are equal up to the transform $t\mapsto -t$ for equal starting points  $\tilde{e}(0)$, and that
these solutions obviously depend analytically from  $\tilde{e}(0)$. We thus conclude from 
 the limit assertion in Lemma \ref{ODE_lem_app} that  $\lim_{t\to-\infty}	\tilde{e}(t)=e(z)$ holds where $D(z)\ne0$ holds.
 As this is a contradiction to $\tilde{D}(\tilde{e}(t))=0$ for  $t\in\mathbb R$, the theorem follows from Lemma \ref{ODE_lem_app-nonvompact}.
	\end{proof}

\end{document}